\documentclass[10pt]{extarticle}

\usepackage[english]{babel}
\usepackage{graphicx}
\usepackage{framed}
\usepackage[normalem]{ulem}
\usepackage{amsmath}
\usepackage{amsthm}
\usepackage{dsfont}
\usepackage{amssymb}
\usepackage{amsfonts}
\usepackage{enumerate}
\usepackage[utf8]{inputenc}
\usepackage[top=1 in,bottom=1 in, left=1 in, right=1 in]{geometry}
\usepackage{mathrsfs}
\usepackage[nottoc, notlof, notlot]{tocbibind}
\usepackage{dsfont}
\usepackage{xcolor}
\usepackage{hyperref}
\usepackage{mathabx}
\usepackage{cases}
\usepackage{stmaryrd}
\usepackage{lineno}
\usepackage{setspace}

\newcommand{\mdiv}{{\rm div}}
\newcommand{\R}{\mathbb{R}}

\newcommand{\E}{\mathbb{E}}

\newcommand{\cC}{\mathcal{C}}
\newcommand{\cF}{\mathcal{F}}

\newcommand{\cP}{\mathcal{P}}
\newcommand{\cL}{\mathcal{L}}
\newcommand{\PP}{\mathbb{P}}

\newtheorem{theorem}{Theorem}[section]
\newtheorem{corollary}{Corollary}[section]
\newtheorem{lemma}{Lemma}[section]
\newtheorem{proposition}{Proposition}[section]

\theoremstyle{definition}
\newtheorem{remark}{Remark}[section]

\theoremstyle{definition}
\newtheorem{definition}{Definition}[section]

\DeclareMathOperator{\Tr}{Tr}

\DeclareMathOperator{\Sp}{S_p}
\DeclareMathOperator{\argmax}{argmax}

\title{Optimal Control of Diffusion Processes with Terminal Constraint in Law}
\author{Samuel Daudin \footnote{PSL Research University, Université Paris-Dauphine, CEREMADE, Place de Lattre de Tassigny, F- 75016 Paris, France}}

\begin{document}


\maketitle

\begin{abstract}
Stochastic optimal control problems with constraints on the probability distribution of the final output are considered. Necessary conditions for optimality in the form of a coupled system of partial differential equations involving a forward Fokker-Planck equation and a backward Hamilton-Jacobi-Bellman equation are proved using convex duality techniques.
\end{abstract}

\section*{Introduction}

This paper is devoted to the study of stochastic optimal control problems with constraints on the law $\cL(X_T)$ of the controlled process at the terminal time. Our problem takes the following form~:

$$\inf_{\alpha_t \in \mathcal{A}} \mathbb{E} \left[ \int_0^T (f_1(t,X_t,\alpha_t) + f_2(t, \mathcal{L}(X_t)))dt + g(\mathcal{L}(X_T)) \right]$$

\noindent under the constraint $\Psi(\mathcal{L}(X_T)) \leq 0$ for the diffusion:

$$dX_t = b(t,X_t,\alpha_t)dt + \sqrt{2}\sigma(t,X_t,\alpha_t)dB_t$$

\noindent with the initial condition given by $\mathcal{L}(X_0) = m_0$ for some $m_0$ in  $\mathcal{P}_2(\R^d)$, the space of probability measures over $\R^d$ with finite second order moment. Here, $f_1 : [0,T] \times \R^d \times A \rightarrow \R$ and $f_2 : [0,T] \times \mathcal{P}_2(\R^d) \rightarrow \R$ are the instantaneous costs, $g : \mathcal{P}_2(\R^d) \rightarrow \R$ is the terminal cost, $\Psi : \mathcal{P}_2(\R^d) \rightarrow \R$ is the final constraint, $b : [0,T] \times \R^d \times A \rightarrow \R^d$ and $\sigma : [0,T]\times \R^d \times A \rightarrow \mathbb{S}_d(\R)$ are respectively the drift and the volatility of the controlled process $X$ and $\alpha$ and is the control process valued in the control space $A$. We look in particular for optimal Markov policies, that is control processes $(\alpha_t)$ which are optimal among all admissible controls and for which there exists some measurable function $\alpha : [0,T] \times \R^d \rightarrow A$ such that, for all $t\in [0,T]$, $\alpha_t = \alpha(t,X_t)$.

We are going to show that optimal Markov policies are related to the solutions of the following system of partial differential equations, where the unknown $(\lambda,\phi,m)$ belongs to $\R^+ \times \mathcal{C}_b^{1,2}([0,T]\times \R^d) \times \mathcal{C}^0([0,T], \mathcal{P}_2(\R^d) ) $~:

\begin{subnumcases}{}
 \displaystyle -\partial_t u(t,x) + H(t,x,Du(t,x),D^2u(t,x)) = \frac{\delta f_2}{\delta m}(t,m(t),x) & in $[0,T]\times \R^d$ \label{hjb} \\
\notag \displaystyle \partial_t m - \mdiv (\partial_p H(t,x,D u(t,x),D^2u(t,x))m) \\
\displaystyle \hspace{50pt} + \sum_{i,j} \partial_{ij}^2 ((\partial_M H(t,x,Du(t,x),D^2u(t,x)))_{ij}m) = 0 & in $[0,T]\times \R^d$ \label{fp} \\
\displaystyle u(T,x) = \lambda \frac{\delta \Psi}{\delta m}(m(T),x) + \frac{\delta g}{\delta m}(m(T),x) \mbox{     in   }  \R^d, 
m(0) = m_0  \\
 \lambda \Psi(m(T)) = 0, \mbox{    }
  \Psi(m(T)) \leq 0,  \mbox{    }
 \lambda \geq 0,
\end{subnumcases}

\noindent where $H(t,x,p,M) := \sup_{a \in A} \left \{ -b(t,x,a).p - \sigma^t\sigma(t,x,a).M- f_1(t,x,a) \right \} $ is the Hamiltonian of the system. The forward equation, Equation \ref{fp} is a Fokker-Planck equation which describes the evolution of the probability distribution $m$ of the optimally controlled process. The backward equation, Equation \ref{hjb} is an Hamilton-Jacobi-Bellman equation satisfied by the adjoint state $u$. The nonnegative parameter $\lambda$ is the Lagrange multiplier associated to the terminal constraint. The forward and backward equations are coupled through the source term for the HJB equation, the terminal condition for the HJB equation and the exclusion condition $\lambda \Psi(m(T)) = 0$.

Our main result, Theorem \ref{MainGeneral} states that, under suitable growth and regularity assumptions, optimal Markov policies $\alpha \in L^0([0,T]\times\R^d,A)$ exist and satisfy~:

$$ \alpha(t,x) \in \argmax_{a \in A} \left \{ -b(t,x,a).Du(t,x) -\sigma^t\sigma(t,x,a).D^2u(t,x)- f_1(t,x,a) \right \}$$

\noindent for some solution $(\lambda, u, m)$ of the above system of PDEs. Notice that we do not a priori require $\Psi$ to be a convex function. When $ \displaystyle \Psi(m)=\int_{\R^d} h(x) m(dx)$ for some function $h : \R^d \rightarrow \R$ and for all $m \in \cP_2(\R^d)$, we say that the constraint is linear. When the costs $f_2$ and $g$ are linear as well we recover the problem of stochastic optimal control under expectation constraint (as in \cite{Bouchard2010}, \cite{Chow2020}, \cite{Pfeiffer2020a}).

Such problems arise in economy and finance when an agent tries to minimize a cost (maximize a utility function) under constraints on the probability distribution of the final output. These types of constraints can take into account the risk given by the dispersion of the cost. There has recently been a surge of interest for this kind of problems. For instance \cite{Guo2019} and \cite{Guo2020} use similar formulations to study respectively the problem of calibration of local-stochastic volatility models and the problem of portfolio allocation with prescribed terminal wealth distribution. Probability constraints of the form $\PP \left[ h(X_T) \leq 0  \right ] \leq 1 - \epsilon $ also fall into our analysis since they can be written as functions of the law $\cL(X_T)$ of $X_T$. In state constrained problems, the constraint is directly imposed on the process $X_T$ and must be satisfied almost-surely. Such constraints might be too stringent or even impossible to satisfy and probability constraints might allow to find controls with a better reward and a controlled probability of failure/success.   

Stochastic control problems with terminal constraints have been extensively studied in the literature. Optimal control problems under stochastic target constraints have been studied in Bouchard, Elie and Imbert \cite{Bouchard2009} using the geometric dynamic programming principle proposed in Soner and Touzi \cite{Soner2002} . In Föllmer and Leukert \cite{Follmer1999}, the authors introduce the notion of quantile hedging to relax almost-sure constraints into probability constraints. In Yong and Zhou \cite{Yong1999} Chapter 3, necessary optimality conditions are proved in the form of a system of forward/backward stochastic differential equations. More recently the problem with constraints on the law of the process has been studied in Pfeiffer \cite{Pfeiffer2020} and in Pfeiffer, Tan and Zhou \cite{Pfeiffer2020a}. In these works, the authors prove that the problem can be reduced to a “standard" problem (without terminal constraint) by adding a term involving $ \lambda^* h$ ---in the case where the constraint has the form $\E \left [h(X_T) \right ] \leq 0$--- to the final cost for some optimal Lagrange multiplier $\lambda^*$.  A dual problem over the Lagrange multipliers associated to the constraints is exhibited using abstract duality results. In Pfeiffer, Tan and Zhou \cite{Pfeiffer2020a}, the authors provide necessary and sufficient optimality conditions for problems with multiple equality and inequality expectation constraints with much less restrictions on the data than we do and in a path dependent framework. However \cite{Pfeiffer2020a} needs to assume some controllability condition (Assumption $3.1.ii$) and works with a compact control set. In our framework, the corresponding controllability condition would be to assume a priori that there exist some control $\alpha$ such that $\E( h(X_T^{\alpha} )) <0$. In our analysis, we are able to prove such controllability condition when $H$ satisfies suitable assumptions.

The novelty of the present work is to provide a framework in which both controllability and existence of strong regular solutions for the Stochastic Control problem can be proved. We also believe that our necessary conditions for optimality can lead to efficient numerical methods using techniques already developed for similar kind of coupled PDE systems as in Achdou and Capuzzo Dolcetta \cite{Achdou2010}. We are also able to handle costs of mean-field type.

Our strategy is to study a relaxed problem which is an optimal control problem for the Fokker-Planck equation and then rely on the regularity of the data to show that optimal controls for the relaxed problem yield optimal controls for the original problem. The relaxed problem is the following~:

\begin{equation*}
\inf_{(m,\omega,W)} \int_0^T \int_{\R^d} L( t,x,\frac{d\omega}{dt \otimes dm}(t,x),\frac{dW}{dt \otimes dm}(t,x) )  dm(t)(x)dt + \int_0^T f_2(t,m(t))dt + g(m(T))
\end{equation*}

\noindent where 
$$L(t,x,q,N):= \sup_{(p,M) \in \R^d \times \mathbb{S}_d(\R)} \left \{ -p.q - M.N - H(t,x,p,M)     \right \} = H^*(t,x,-q,-N)$$
and the infimum is taken over the triples $(m,\omega,W) \in \mathcal{C}^0([0,T],\mathcal{P}_1(\R^d)) \times \mathcal{M}([0,T]\times \R^d,\R^d) \times  \mathcal{M}([0,T]\times \R^d,\mathbb{S}_d(\R))$ for which $\omega$ and $W$ are absolutely continuous with respect to $m(t) \otimes dt$ and $(m,\omega,W)$ satisfy in the sense of distribution the Fokker-Planck equation: 
$$ \partial_t m + \mdiv \omega -\sum_{i,j} \partial^2_{ij} W_{ij} = 0 $$ 
together with the initial condition $m(0)=m_0$ and the terminal constraint $\Psi(m(T)) \leq 0$. Notice that here and in the following, we denote by $\mathbb{S}_d(\R)$ the space of symmetric matrices of size $d$, endowed with the inner product $M.N:= \Tr(MN)$ and by $\mathcal{M}([0,T] \times \R^d, \R^d)$ (respectively by $\mathcal{M}([0,T]\times \R^d,\mathbb{S}_d(\R))$)  the space of  $\R^d$-valued (respectively $\mathbb{S}_d(\R)$-valued) Borel measures on $[0,T] \times \R^d$ with finite total variation.

In order to study the relaxed problem, we rely on duality techniques that originated in the theory of Optimal Transport (see  \cite{Rachev1998}, \cite{Villani2003}, \cite{Villani2007} and \cite{Benamou2000}) and were further developed in the theory of Mean Field Games. Indeed, when the game has a potential structure --- see for instance Lasry, Lions \cite{Lasry2007}, Cardaliaguet, Graber, Porretta and Tonon \cite{Cardaliaguet2015}, Briani and Cardaliaguet \cite{Briani2018} and Orrieri, Porretta and Savaré \cite{Orrieri2019} ---  the system of partial differential equations which describes the distribution of the players and the value function of a typical infinitesimal player can be obtained as optimality conditions for an optimal control problem for the Fokker-Planck equation. In this framework, the necessary conditions are obtained through convex duality techniques, using generally  the Fenchel-Rockafellar theorem as in \cite{Cardaliaguet2015}, \cite{Briani2018}  or the Von-Neumann theorem as in \cite{Orrieri2019}. We follow this path and --- when the final constraint as well as the costs $f_2$ and $g$ are linear --- we are able to exhibit a dual problem, which is an optimal control problem for the HJB equation involving the Lagrange multiplier $\lambda \in \R^+$ associated to the terminal constraint. It takes the following form~:

\begin{equation*}
\sup_{(\lambda,\phi)} \int_{\R^d} \phi(0,x) dm_0(x) 
\end{equation*}

\noindent where the supremum runs over the couples $(\lambda,\phi) \in \R^+ \times \mathcal{C}_b^{1,2}([0,T]\times \R^d)$ satisfying

\begin{equation*}
\left \{
\begin{array}{ll}
-\partial_t \phi(t,x) + H(t,x,D\phi(t,x),D^2 \phi(t,x)) \leq f_2'(t,x) \mbox{ in     } [0,T]\times \R^d  \\
\phi(T,x) \leq  \lambda h(x) + g'(x) \mbox{ in   } \R^d
\end{array}
\right.
\end{equation*}
and where $f_2': [0,T] \times \R^d \rightarrow \R$ and $g':\R^d \rightarrow \R$ are such that $\displaystyle f_2(t,m) = \int_{\R^d} f_2'(t,x)dm(x)$ and $\displaystyle g(m) = \int_{\R^d} g'(x)dm(x)$.

The necessary conditions for optimality then follow from the lack of duality gap between the relaxed and the dual problems. We can then address more general constraints $\Psi : \mathcal{P}_2(\R^d) \rightarrow \R$ and costs $f_2 : [0,T] \times \mathcal{P}_2 (\R^d) \rightarrow \R$, $g: \mathcal{P}_2(\R^d)$ by “linearizing" the costs and the constraint around solutions of the relaxed problem.

Using convex duality techniques to solve optimal control problems for diffusion processes is of course not new. It can be traced back at least to Fleming and Vermes \cite{Fleming1989} where the philosophy is very close to ours. In Tan and Touzi, \cite{Tan2013} the authors extend the usual Monge-Kantorovitch optimal transportation problem to a stochastic framework. The mass is transported along a continuous semimartingale and the initial and terminal distributions are prescribed. Studying optimal control problems for the Fokker-Planck equation in order to understand the stochastic control problem is less common and it seems adapted to problems where the constraints only act on the law of the process. We refer to the works of Blaquière \cite{Blaquiere1992} and more recently, Mikami \cite{Mikami2015} and Mikami and Thieullen \cite{Mikami2006} where similar approaches are developed in connection with the so-called Schrödinger problem. This approach has been followed recently by Guo, Loeper and Wang \cite{Guo2019} and Guo, Loeper, Langrené and Ning \cite{Guo2020} for problems with various expectation constraints. In both papers, the authors show that their original problem is in duality with a problem of optimal control of sub-solutions of an HJB equation. This dual problem is solved numerically. Our relaxation is in the spirit of classical works in convex analysis (see \cite{Ekeland1999a}) but usually probabilists prefer to study another relaxation of the initial problem through the martingale problem (see Stroock and Varadhan \cite{Stroock1997}), as in El Karoui, Jeanblanc-Picqué and Nguyen \cite{ElKaroui1987} or Lacker \cite{Lacker2015}. These different ways to relax the initial problem are, of course, connected and the correspondences between the diffusion processes, the martingale problem and the Fokker-Planck equation are now well established starting from the seminal work of \cite{Stroock1997} and more recently Figalli \cite{Figalli2008} and Trevisan \cite{Trevisan2016}.

Under very general assumptions, as in \cite{Fleming1989}, one is usually able to see that the original problem is in duality with a problem of optimal control of the HJB equation. However, existence of solutions for this dual problem is much harder to come by and requires particular structural conditions. Essentially, the dual problem has a solution if the Hamilton-Jacobi-Bellman equation admits a regular solution. This is of course rather difficult to obtain. Regularity results for the Hamilton-Jacobi-Bellman equation where the control appears in the volatility as in Fleming and Soner \cite{Fleming2006} Chapter $IV.4$, usually rely upon three things~: the regularity of the coefficients of the diffusion and of the costs functionals, the compactness of the control set and finally the uniform parabolicity of the equation. The last point means that there must be some $\Lambda^- >0$ such that the volatility coefficient satisfies (uniformly in the time/state/control variables)  $\sigma^t\sigma \geq \Lambda^- I_d$.  

In studying terminal constraints, compact control sets are not satisfactory since we would not be able to show, in full generality, that the constraint can indeed be reached with a finite cost. Part of the challenge of the paper is to find a framework in which the process is sufficiently “controllable" but the HJB equation is still solvable. For that we need to impose restrictions on the coefficients.

In particular, we require some growth assumptions on the Hamiltonians and its derivatives. This allows us to use the weak Bernstein method as in Ishii and Lions \cite{Ishii1990}, Barles \cite{Barles1991}, Lions and Souganidis \cite{Lions2005} and Armstrong and Cardaliaguet \cite{Armstrong2018} (among others) to prove that the viscosity solution of the HJB equation is Lipschitz in time and space. 

As it is well-known, controllability for such systems is related to the coercivity of the Hamiltonian $H$ in the momentum variable.  As we will show, imposing a strictly super-linear polynomial growth (in $p$) for $H(t,x,p,0) : = \sup_{a \in A} -b(t,x,a).p - f_1(t,x,a) $ allows to show that the agent can take (with a relaxed control) any instantaneous drift without paying too big a cost.

The rest of the paper is organized as follows : in Section \ref{sec: Statement of the problem and Hypothesis} we present our assumptions and the precise statement of the problem. We also give our main results there. In Section \ref{sec: A relaxed Problem}  we introduce and study the problem of optimal control of the Fokker-Planck equation. Our main results, Theorems \ref{MainGeneral} and \ref{MainConv} are then proved in Section \ref{sec: Proof of the Main Theorem}. Finally we give in Section \ref{sec: HJB} a detailed study of the Hamilton-Jacobi-Bellman equation which is crucial to our analysis.


\section{Main results}

\label{sec: Statement of the problem and Hypothesis}

In this section we first present our notations and our standing assumptions. Then we briefly discuss some properties of the Lagrangian $L$ and finally we state our main results.

\subsection{Notations and functional spaces}

The $d$-dimensional euclidean space is denoted by $\R^d$ and the space of real matrices of size $d$ by $\mathbb{M}_{d}(\R)$. The space of symmetric matrices of size $d \times d$ is denoted by $\mathbb{S}_d(\R)$.  The subset of $\mathbb{S}_d(\R)$ consisting of positive symmetric matrices is denoted by $\mathbb{S}^+_d(\R)$ and $\mathbb{S}^{++}_d(\R)$ is the subset of $\mathbb{S}_d(\R)$ consisting of definite-positive symmetric matrices. Recall that $\mathbb{S}^{++}_d(\R)$ is endowed with a smooth (analytic) square root~: $\sqrt{.} : \mathbb{S}^{++}_d(\R) \rightarrow \mathbb{S}^{++}_d(\R) $ (see for instance \cite{Stroock1997} Lemma 5.2.1). Sometimes we will use $\Sp(A)$ to denote the set of eigenvalues of a square matrix $A$. The euclidean space $\R^d$ is endowed with its canonical scalar product : $x.y := \sum_{i=1}^d x_iy_i $ and the associated norm $|x|^2 := \sum_{i=1}^d x_i^2$. The space $\mathbb{M}_d(\R)$ is endowed with its canonical scalar product : $M.N := \Tr (^tMN) $ and the associated norm $|M|^2 := \Tr(^tMM)$ where $\Tr(M)$ is the trace of $M$ and $^tM$ is the transpose of $M$. Sometimes we will use the operator norm on $\mathbb{M}_d(\R)$ : $\left \vvvert M \right \vvvert := \sup_{x \in \R^d} \frac{|Mx|}{|x|}$.
For two real numbers $r_1$ and $r_2$, $r_1 \wedge r_2$ is the minimum of $r_1$ and $r_2$ and $r_1 \vee r_2$ is the maximum of $r_1$ and $r_2$.
If $\eta$ is a $\sigma$-finite positive measure on a measurable space $(\Omega, \cF)$ and if $\mu$ is a $\sigma$-finite vector measure on $(\Omega, \cF)$, we write $\mu << \eta$ if $\mu$ is absolutely continuous with respect to $\eta$ and we write $\frac{d\mu}{d\eta} \in L^1(\eta)$ for the Radon-Nikodym derivative of $\mu$ with respect to $\eta$.
If $E$ is a locally compact, complete, separable metric space and $l \geq 1$ is an integer, $\mathcal{C}_0(E, \R^l) $ is the space of $\R^l$-valued continuous functions on $X$, vanishing at infinity. It is endowed with the topology of uniform convergence. Its topological dual $\left( \mathcal{C}_0(E, \R^l) \right) ^*$ can be identified thanks to Riesz theorem as the space $\mathcal{M}(E, \R^l) $ of $\R^l$-valued Borel measures with finite total variation on $E$, normed by total variation. We will often consider the weak-* topology on $ \mathcal{M}(E, \R^l)  $. When $l =1$ we simply note $\mathcal{C}_0(E) $  and $ \mathcal{M}(E)$. $\mathcal{M}^+(E) \subset \mathcal{M}(E)$ is the cone of finite non-negative measures.
The set of Borel probability measures over $E$ is denoted by $\mathcal{P}(E)$.
If $r \geq 1$, $\mathcal{P}_r(E)$ is the set of Borel probability measures over $E$ with finite moment of order $r$. It is endowed with the topology given by the Wasserstein distance $d_r$ of order $r$. If $X$ is a random variable taking values into $(\R^d, \mathcal{B}(\R^d))$, its law is denoted by $\mathcal{L}(X) \in \cP(\R^d)$.
We say that $U : \mathcal{P}_1(\R^d) \rightarrow \R$ if $\mathcal{C}^1$ if there is a bounded continuous function $\frac{\delta U}{\delta m} :\mathcal{P}_1(\R^d) \times \R^d \rightarrow \R $ such that, for any $m_1,m_2 \in \cP_1(\R^d)$,

$$ U(m_1) - U(m_2) = \int_0^1 \int_{\R^d} \frac{\delta U}{\delta m}  ( (1-t)m_2 + tm_1,x)(m_1-m_2)(dx)dt. $$

This derivative is defined up to an additive constant and we use the standard normalization convention~: $ \displaystyle \int_{\R^d} \frac{\delta U}{\delta m}(m,x)m(dx) = 0$. See \cite{Carmona2018a} for details on the notion(s) of derivatives in the space of measures.

We consider a finite, fixed horizon $T >0$.
The set of continuous functions from $[0,T]$ to $\cP(\R^d)$ and from $[0,T]$ to $\cP_r(\R^d)$ for $r\geq 1$ are respectively denoted by $\mathcal{C}^0([0,T], \cP(\R^d))$ and by $\mathcal{C}^0([0,T], \cP_r(\R^d))$.
The space of measurable functions defined on $[0,T] \times \R^d$ with values into the measurable space $Y$ is denoted by $L^0([0,T] \times \R^d, Y)$.
If $ u : [0,T] \times \R^d \rightarrow \R $ is sufficiently smooth, $Du : [0,T] \times \R^d \rightarrow \R^d$ and $D^2u : [0,T] \times \R^d \rightarrow \mathbb{S}^d(\R)$ denote respectively the differential and the Hessian of $u$ with respect to the space variable $x$.
The space of continuous functions $u$ on $[0,T] \times \R^d$ for which $\partial_t u$, $Du$ and $D^2u$ exist and are continuous is denoted by $\cC^{1,2}([0,T] \times \R^d)$ and $\cC^{1,2}_b([0,T] \times \R^d)$ is the subspace of $\cC^{1,2}([0,T] \times \R^d)$ consisting of functions $u$ for which $u$, $\partial_tu$, $Du$ and $D^2 u$ are bounded.
If $n \in \mathbb{N}^*$ and $\alpha \in (0,1)$, $\mathcal{C}_b^{n+ \alpha}(\R^d)$ is the space of bounded continuous real functions on $\R^d$ for which the first $n$-derivatives are continuous and bounded and the $n$-th derivative is $\alpha$-Hölder continuous. We say that $\phi :[0,T] \times \R^d \rightarrow \R^d$ is in $\cC_b^{\frac{n+\alpha}{2},n+\alpha}([0,T]\times \R^d)$ if $\phi$ is continuous in both variables together with all derivatives $D_t^rD_x^s\phi$ with $2r+s \leq n$. Moreover, $\|\phi \|_{ \frac{n+\alpha}{2},n+\alpha}$ is bounded where
\begin{align*}
\|\phi\|_{ \frac{n+\alpha}{2},n+\alpha} &:= \sum_{2r+s \leq n} \|D_t^rD_x^s\phi \|_{\infty} +\sum_{2r+s=n} \sup_{t \in [0,T]} \|D_t^rD_x^s \phi(t,.) \|_{\alpha} \\
&+ \sum_{0<n+\alpha -2r-s<2} \sup_{x \in \R^d} \|D_t^r D_x^s \phi(.,x) \|_{\frac{n+\alpha-2r-s}{2}}.
\end{align*}

\subsection{Assumptions}

In all the following, $A$ is a closed subset of an euclidean space, $T > 0$ is a finite horizon and $r_2 \geq r_1 > 1$ are two parameters. The conjugate exponents of $r_1$and $r_2$ are respectively denoted by $r_1^*$ and $r_2^*$. The data are:

$$(b, \sigma,f_1) : [0,T] \times \R^d \times A \rightarrow \R^d \times \mathbb{S}_+^d(\R) \times \R, $$
$$f_2:[0,T] \times \mathcal{P}_1(\R^d) \rightarrow \R, $$
$$g : \mathcal{P}_1(\R^d)  \rightarrow \R,$$
$$\Psi : \cP_1(\R^d) \rightarrow \R, $$
$$m_0 \in \mathcal{P}_{r_1^* \vee 2}(\R^d).$$

We define the Hamiltonian of the system, for all $(t,x,p,M) \in [0,T] \times \R^d \times \R^d \times \mathbb{S}^d(\R) $~:

$$ H(t,x,p,M) = \sup_{a \in A} \left\lbrace -b(t,x,a).p -\sigma(t,x,a) ^t\sigma(t,x,a) .M  - f_1(t,x,a) \right\rbrace $$

\begin{enumerate}

\item \textbf{Assumptions on $b,\sigma,f_1,f_2$ and $g$}

\begin{enumerate}
\item For all $R >0$, $b$, $\sigma$ and $f_1$ as well as the partial derivatives  $\partial_x b$, $\partial_tb$, $\partial_{xx}^2b$,  $\partial_x \sigma$, $\partial_t \sigma$, $\partial_{xx}^2 \sigma$, $\partial_x f_1$, $\partial_t f_1$, $\partial_{xx}^2 f_1$, are continuous and bounded on $[0,T] \times \R^d \times \left( A \cap B(0,R) \right) $ ;  $\partial_{x} b$ and $\partial_x \sigma$ are globally bounded.
\item $b$ has at most a linear growth and $\sigma$ satisfies $\Lambda^- I_d \leq \sigma^t\sigma(t,x,a) \leq \Lambda^+ I_d$ for some $\Lambda^+ \geq \Lambda^- >0$ uniformly in $(t,x,a)$.
\label{Uniformellipticity}
\item $f_1$ is continuous and coercive with respect to $a$: there is $\delta >0$ and $C_1,C_2>0$ such that, for all $(t,x,a)$, $f_1(t,x,a) \geq C_1|a|^{1+\delta} -C_2$.
\item $f_2$ is continuous, bounded and has one linear derivative in $m$. The first order functional derivative $\displaystyle \frac{\delta f_2}{\delta m}: [0,T] \times \mathcal{P}_1(\R^d) \times \R^d \rightarrow \R$ is globally Lipschitz continuous, bounded and $\displaystyle x \rightarrow \frac{\delta f_2}{\delta m}(t,m,x)$ belongs to $\mathcal{C}_b^{3+\alpha}(\R^d)$ with bounds uniform in $(t,m)$.
\item $g$ is continuous, bounded and has one functional derivative in $m$ such that $ \displaystyle x \rightarrow  \frac{\delta g}{\delta m}(m,x) $ belongs to $\mathcal{C}_b^{3+\alpha}(\R^d)$ with bounds uniform in $m$.
\label{HoldergAss}
\end{enumerate}
\label{IA}

\item \textbf{Assumptions on the Hamiltonian}

\label{HAss}
\begin{enumerate}

\item $H$ is $C^1$ in $(t,x,p,M)$. The partial derivatives $\partial_xH$, $\partial_p H$ and $\partial_M H$ are Lipschitz in $[0,T] \times \R^d \times B(0,R) \times B(0,R)$ for all $R>0$.
\label{C1HAss}
\item  There is some $\alpha_1,\alpha_2 >0$ and $C_H >0$ such that, for all $(t,x,p) \in [0,T] \times \R^d \times \R^d$,
$$\alpha_1 |p|^{r_1} - C_H \leq H(t,x,p,0) \leq \alpha_2 |p|^{r_2} +C_H. $$
\label{GrA}
\item $\partial_t H(t,x,p,M)$ is bounded over $[0,T] \times \R^d \times \R^d \times \mathbb{S}_d(\R)$.
\label{DerGrDtHA}
\item There is some positive constant $C_{D_pH}$ and an exponent $\nu \geq 1$ such that, for all $(t,x,p,M) \in [0,T] \times \R^d \times \R^d \times \mathbb{S}_d(\R)$
$$|D_pH(t,x,p,M)| \leq C_{D_pH}(1 + |p|^\nu).$$
\label{DerGrDpHA}
\item $D_xH$ is uniformly in $(t,x,p) \in [0,T] \times \R^d \times \R^d$ Lipschitz continuous in $M$.
\label{DerGrDxMHA}
\item 
\begin{enumerate}
\item \textbf{Either} $f_2 =0$ and the limit
$ \displaystyle \lim_{|p| \rightarrow +\infty} \frac{|p|^2 + \partial_xH(t,x,p,0).p}{H^2(t,x,p,0)} = 0 $ holds uniformly in $(t,x) \in [0,T] \times \R^d$
\label{DerGrDxHA1} 
\item \textbf{or} $f_2 \neq 0$ and there is some $C_{D_xH} >0$ such that $\displaystyle |D_xH(t,x,p,0) | \leq C_{D_xH} (1+ |p|)$.
\label{DerGrDxHA2}

\end{enumerate}
\label{DerGrDxHA}

\end{enumerate}

\item \textbf{Assumptions on the constraint $\Psi$}
\label{FCg}
\begin{enumerate}
\item $\Psi$ is continuous and admits a functional derivative such that $ \displaystyle x \rightarrow  \frac{\delta \Psi}{\delta m}(m,x) $ belongs to $\mathcal{C}^{3+\alpha}_b(\R^d)$ with bounds uniform in $m$.

\item There is at least one $m \in \mathcal{P}_1(\R^d)$ such that $\Psi(m)<0$.

\item For all $m \in \mathcal{P}_1(\R^d)$ such that $\Psi(m)=0$ there exists $x_0 \in \R^d$ such that $\displaystyle \frac{\delta \Psi}{\delta m}(m,x_0) <0$.
\label{TransCondU}
\end{enumerate}

\end{enumerate}

\begin{remark} Assumption \ref{IA} is sufficient to uniquely define the controlled process $X^{\alpha}$ for  any control $\alpha \in \mathcal{A}$ (see below for the definitions). If $A$ were compact with $f_2=0$, we would be in the setting of \cite{Fleming2006} Chapter IV.4 and these assumptions would guarantee the existence of a smooth value function (in $\cC_b^{1,2} ([0,T]\times \R^d)$). 
\end{remark}

\begin{remark} The upper bound in Assumption \ref{GrA} is a coercivity assumption on the cost $f_1$ relatively to the drift $b$. Taking the definition of $H$, we see that it is equivalent to ask that, for all $(t,x,a) \in [0,T] \times \R^d \times A$, $f_1(t,x,a) \geq \alpha_2' |b(t,x,a)|^{r_2^*} - C_H, $ for some $\alpha_2' >0$. It will be a source of compactness throughout the paper. The lower bound in Assumption \ref{GrA} is a “weak"-controllability condition and we will discuss it further in Lemma \ref{GrH_1 explained}.
\end{remark}

\begin{remark} Using the Envelope theorem (see for instance \cite{Milgrom2002}) we see that $H$ being $\mathcal{C}^1$ ---Assumption \ref{C1HAss}--- in the $p,M$-variables implies that, for any $a(t,x,p,M) \in A$ such that $H(t,x,p,M) $ $= -b(t,x,a(t,x,p,M)).p -\sigma ^t\sigma (t,x,a(t,x,p,M)) - f_1(t,x,a(t,x,p,M))$ we get $\partial_pH(t,x,p,M) =$ $-b(t,x,a(t,x,p,M))$ and  

\noindent $\partial_MH(t,x,p,M) =$ $-\sigma ^t\sigma (t,x, a(t,x,p,M))$. Consequently, drift and volatility must agree on potentially different optimal controls with common values $-\partial_pH(t,x,p,M) $ and $\sqrt{-\partial_MH(t,x,p,M) }$ respectively. Notice that the growth conditions on the cost $f_1$ and the drift $b$ ensure that for any $(t,x,p,M) \in [0,T] \times \R^d \times \R^d \times \mathbb{S}_d(\R) $, there exists at least one such $a(t,x,p,M)$ in $A$.
\end{remark}

\begin{remark} Using the envelope theorem and the uniform ellipticity condition in Assumption \ref{Uniformellipticity} we see that for all $(t,x,p,M)$, $\displaystyle \Lambda^-I_d \leq -\partial_M H(t,x,p,M) \leq \Lambda^+I_d$, a fact that we will repeatedly used throughout the paper.
\end{remark} 

\begin{remark} We use (the restrictive) Assumptions \ref{DerGrDtHA}, \ref{DerGrDpHA}, \ref{DerGrDxMHA}, \ref{DerGrDxHA} in order to find Lipschitz estimates for the solution of the Hamilton-Jacobi-Bellman equation and to deduce that it is well-posed in $\cC_b^{1,2} ([0,T]\times \R^d)$. Assumptions \ref{C1HAss} is then sufficient to show that the solution is actually in $\cC_b^{\frac{3+\alpha}{2},3+\alpha}([0,T] \times \R^d)$. When Assumption \ref{GrA} hold, Assumption \ref{DerGrDxHA2} is stronger than Assumption \ref{DerGrDxHA1} but we use it to find Lipschitz estimates which are independent from the time regularity of the source term of the HJB equation. 
\end{remark}

\begin{remark} Assumption \ref{TransCondU} is a tranversality condition. When $\Psi$ is convex, this assumption is equivalent to the existence of some probability measure $m \in \cP_1(\R^d)$ such that $\Psi(m) <0$.
\end{remark}

The following observations will be useful in order to translate the properties of the Hamiltonian $H$ into properties of the Lagrangian $L$ defined for all $(t,x,q,N) \in [0,T] \times \R^d \times \R^d \times \mathbb{S}_d(\R)$ by

$$L(t,x,q,M) := \sup_{(p,M) \in \R^d \times \mathbb{S}_d(\R) } \left\lbrace -p.q -M.N -H(t,x,p,M) \right \rbrace. $$

\noindent Taking convex conjugates in \ref{GrA} we see that this assumption can be reformulated in terms of $L$: for all $(t,x,q) \in [0,T] \times \R^d \times \R^d$,

\begin{equation}
\label{EstimateH'}
\alpha_2' |q|^{r_2^*} - C_H \leq L(t,x,q,0) \leq \alpha_1' |q|^{r_1^*} + C_H
\end{equation}

\noindent where, for $i=1,2$, $\displaystyle  \alpha_i' = \alpha_i^{-\frac{1}{r_i-1}}~(~r_i~-~1~)~r_i^{\frac{-r_i}{r_i-1}}$ and $r_i^* = \frac{r_i}{r_i-1}$ is the conjugate exponent of $r_i$.

Throughout the article, the following dual representation for $L$ will be useful. 

\begin{lemma}

\label{dual representation for $H_1^*$ and $H_2^*$}

Under Assumption \ref{IA} above, for all $(t,x,p,M) \in [0,T] \times \R^d \times \R^d \times \mathbb{S}_d(\R)$, $L(t,x,p,M) <+\infty$ if and only if there is $\mathbf{q}_A \in \mathcal{P}_1(A)$ such that $ \displaystyle \int_A b(t,x,a) d\mathbf{q}_A(a) = q$ and $\displaystyle \int_A \sigma ^t\sigma(t,x,a) d\mathbf{q}_A(a) = N$ and in this case

$$L(t,x,q,N) = \min_{ \mathbf{q}_A } \int_A f(t,x,a)  d\mathbf{q}_A(a)  $$

\noindent where the minimum is taken over the $\mathbf{q}_A \in \mathcal{P}_1(A)$ such that $ \displaystyle \int_A b(t,x,a) d\mathbf{q}_A(a) = q$ and $\displaystyle \int_A \sigma ^t\sigma(t,x,a) d\mathbf{q}_A(a)$ $= N$.
\end{lemma} 

\begin{proof}
It is elementary to show that for all $(t,x,p,M) \in [0,T] \times \R^d \times \R^d \times \mathbb{S}_d(\R)$,
$$ H(t,x,p,M) = \sup_{\mathbf{q}_A \in \mathcal{P}_1(A)} \left \{ \int_A (-b(t,x,a).p-\sigma^t\sigma(t,x,a) - f_1(t,x,a))d\mathbf{q}_A (a) \right \} $$
and therefore $L$ reads as follows for all $(t,x,q,N) \in [0,T] \times \R^d \times \R^d \times \mathbb{S}_d(\R)$,

$$L (t,x,q,N) = \sup_{p,M} \left \{ \inf_{\mathbf{q}_A \in \mathcal{P}_1(A)} \left \{ p.q + M.N + \int_A (f(t,x,a) - b(t,x,a).p - \sigma^t\sigma(t,x,a).M ) d\mathbf{q}_A (a) \right \} \right \}.$$
The result follows by exchanging the “sup" and the “inf". To this end we use Von Neumann Theorem \ref{VN} in the Appendix. The coercivity of $f_1$ as well as results of \cite{Ambrosio2005} (Proposition 7.1.5) about the lower semicontinuity of functions defined on the space of probability measures allow to ensure that the use of the minmax theorem is licit.
\end{proof}

\noindent From this dual representation we can see that the lower-bound on $H(t,x,p,0)$ ---or equivalently the upper-bound on $L(t,x,q,0)$--- is a “weak"-controllability condition. It ensures that the agent can take any drift with a relaxed (i.e measure-valued) control without paying more than the $r_1^*$-power of the drift~:

\begin{lemma}
Fix $(t,x) \in [0,T] \times \R^d$. It holds that $H(t,x,p,0) \geq \alpha_1 |p|^{r_1} -C_H$ for all $p \in \R^d$ if and only if, for all $q \in \R^d$ there exists $\mathbf{q}_{A} \in \mathcal{P}_1(A)$ such that $\displaystyle  q = \int_{A} b(t,x,a) d\mathbf{q}_A(a) $ and $\displaystyle   \int_{A} f_1(t,x,a) d\mathbf{q}_A(a) \leq \alpha _1' |q|^{r_1^*} + C_H$.
\label{GrH_1 explained}
\end{lemma}

\noindent For example, the growth condition on $H$ is satisfied if $Conv(Im(b(t,x,.)) = \R^d$ for all $(t,x) \in [0,T] \times \R^d$ and for all $(t,x,a) \in [0,T] \times \R^d \times A$, $ \alpha_2' |b(t,x,a)|^{r_2^*} - C_H \leq f_1(t,x,a) \leq \alpha_1' |b(t,x,a)|^{r_1^*} + C_H. $

\subsection{Main results}

\noindent Throughout the article, we consider a fixed filtered probability space $(\Omega, \mathcal{F}, \mathbb{F}, \mathbb{P})$ with $\mathbb{F} = (\mathcal{F}_t)_{t\geq 0} $ satisfying the usual conditions and supporting an adapted, standard $d$-dimensional Brownian motion $(B_t)_{t\geq 0}$. We fix a $\mathcal{F}_0$-measurable random variable $X_0$, independent of $(B_t)$ and such that $X_0$ belongs to $L^{r_1^*\vee 2}(\mathbb{P})$. The control process  $\alpha = (\alpha_t)_{t \geq 0}$ is a progressively measurable process valued in $A$ with finite $\mathbb{L}^2(\Omega \times [0,T))$-norm. We denote by $\mathcal{A}$ the set of control processes.  From the Cauchy-Lipschitz theorem, we know that for every $\alpha \in \mathcal{A}$, there exists a unique $\mathbb{F}$-adapted process $X^{\alpha}$ satisfying :

$$ dX_t = b(t,X_t,\alpha_t)dt + \sqrt{2}\sigma(t,X_t,\alpha_t) dB_t $$

\noindent with the initial condition $X_0^{\alpha} = X_0$. A particular class of controls which is of interest is the one of Markovian controls (or Markov policies). A control process $\alpha$ is a Markovian control if there is a measurable function $\alpha : [0,T] \times \R^d \rightarrow $ such that, for all $t \in [0,T]$, $\alpha_t = \alpha(t,X_t^{\alpha})$. We now introduce the cost functional $J_{SP} : \mathcal{A} \rightarrow \R \cup \{+\infty\} $

$$J_{SP} (\alpha) := \E \left[ \int_0^T \left( f_1(s,X_s^{\alpha},\alpha_s) + f_2(s,\mathcal{L}(X_s^{\alpha})) \right) ds + g(\mathcal{L}(X_T^{\alpha})) \right]. $$

\noindent The optimal control problem we are interested in is to minimize $J_{SP} (\alpha)$ over $\alpha \in \mathcal{A} $ under the constraint $\Psi(\cL(X_T)) \leq 0$.

\noindent If there exists a continuous function $h : \R^d \rightarrow \R $ such that, for all $m \in \cP_1(\R^d) $,  $\displaystyle \Psi(m)  = \int_{\R^d} h(x) dm(x) $ then will say that the final constraint is \textbf{linear}. We define the set of admissible controls $\mathcal{U}_{ad}$ 

$$ \mathcal{U}_{ad} := \left \{ \alpha \in \mathcal{A} \mbox{  :  } \Psi(\cL(X_T^{\alpha}) ) \leq 0 \mbox{  and  } J_{SP} (\alpha) < +\infty  \right \}. $$

\noindent The problem in strong formulation is thus :

\begin{equation}
\inf_{\alpha \in \mathcal{U}_{ad} } J_{SP}(\alpha). 
\tag{SP}
\label{SP}
\end{equation}

\noindent The fact that  $\mathcal{U}_{ad}$ is not empty is not trivial in itself but in our setting we will show that there are indeed admissible controls. Our results are the following :

\begin{theorem}[HJB equation] 
Take $g' \in \mathcal{C}_b^{3+\alpha}(\R^d)$ and $f_2' \in \mathcal{C}_b ([0,T], \mathcal{C}_b^{3+\alpha}(\R^d))$ such that $t \rightarrow f_2'(t,x) \in \mathcal{C}^{\alpha}([0,T])$ for all $x \in \R^d$ with bounds uniform in $x$. Assume further that Assumptions \ref{IA} and \ref{HAss} hold with \ref{DerGrDxHA1} in force if $f_2' =0$ and  \ref{DerGrDxHA2} in force if $f_2' \neq 0$. Then the Hamilton-Jacobi-Bellman equation

\begin{equation*}
\left \{
\begin{array}{ll}
\displaystyle -\partial_t \phi(t,x) + H(t,x, D\phi(t,x),D^2\phi(t,x)) = f_2'(t,x) & \mbox{in } [0,T]\times \R^d\\
\phi(T,x) = g'(x) & \mbox{in } \R^d
\end{array}
\right.
\end{equation*}

\noindent admits a unique strong solution $\phi \in \cC_b^{ \frac{3+\alpha}{2}, 3+\alpha}([0,T] \times \R^d)$.
\label{MainHJB}
\end{theorem}

\begin{theorem}[General Constraint] Under Assumptions \ref{IA}, \ref{HAss} and \ref{FCg}, there exist optimal Markov policies. Moreover, if  $(\alpha_t) \in \mathcal{A}$ is an optimal Markov policy, then there exists $(\lambda, \phi, m) \in \R^+ \times \mathcal{C}_b^{1,2}([0,T]\times \R^d)) \times \mathcal{C}^0([0,T], \mathcal{P}_2(\R^d) ) $ such that, for $m(t) \otimes dt$-almost all $(t,x)$ in $[0,T]\times \R^d$

\begin{align}
  \label{optimalityalpha} H(t,x,D\phi(t,x),D^2\phi(t,x)) = &-b(t,x, \alpha(t,x)) . D\phi(t,x) -\sigma^t\sigma(t,x,\alpha(s,x)).D^2\phi(t,x) \\
& - f_1(t,x, \alpha(t,x)) \notag
\end{align}

\noindent and $(\lambda, \phi, m)$ satisfies the system of optimality conditions :

\begin{equation} 
\left \{
\begin{array}{ll}
\displaystyle -\partial_t \phi(t,x) + H(t,x,D\phi(t,x),D^2\phi(t,x)) = \frac{\delta f_2}{\delta m}(t,m(t),x)  & \mbox{in } [0,T]\times \R^d\\
\displaystyle \partial_t m - \mdiv (\partial_p H(t,x,D \phi(t,x),D^2\phi(t,x))m) \\
\displaystyle \hspace{50pt}+ \sum_{i,j} \partial_{ij}^2 ((\partial_M H(t,x,D\phi(t,x),D^2\phi(t,x)))_{ij}m) = 0 & \mbox{in } [0,T]\times \R^d \\
\displaystyle \phi(T,x) = \lambda \frac{\delta \Psi}{\delta m}(m(T),x) + \frac{\delta g}{\delta m}(m(T),x)  \mbox{      in      } \R^d \mbox{,       }
\displaystyle m(0) = m_0, \\
\displaystyle \lambda \Psi(m(T)) = 0 \mbox{,        }
\displaystyle  \Psi(m(T)) \leq 0 \mbox{,            }
\displaystyle \lambda \geq 0. \\
\end{array}
\right.
\tag{OC}
\label{OC}
\end{equation}

\noindent Furthermore, $m(t)$ is actually the law of the optimally controlled process $X_t^{\alpha}$ and the value of the problem -denoted by $V_{SP}(X_0)$- is given by 
$$V_{SP}(X_0) := \displaystyle \inf_{\alpha \in \mathcal{U}_{ad} } J_{SP}(\alpha) = \int_{\R^d} \phi(0,x) dm_0(x) + \int_0^T f_2(t,m(t))dt + g(m(T)).$$
\label{MainGeneral}
\end{theorem}

When the constraint and the costs $f_2$ and $g$ are convex in the measure variable, we are able to show that the conditions are also sufficient :

\begin{theorem}[Convex constraint and convex costs]
If $\Psi$, $f_2$ and $g$ are convex in the measure argument and Assumptions \ref{IA}, \ref{HAss} and \ref{FCg} hold, then the conditions of Theorem \ref{MainGeneral} are also sufficient conditions: if $\alpha \in L^0([0,T] \times \R^d,A)$ satisfies \ref{optimalityalpha} for some  $(\phi,m,\lambda)$ satisfying \ref{OC} then the SDE 
$$dX_t =b(t,X_t,\alpha(t,X_t))dt + \sqrt{2}\sigma(t,X_t,\alpha(t,X_t))dB_t$$
starting from $X_0$ has unique strong solution $X_t$, it holds that $m(t) = \mathcal{L}(X_t)$ and $\alpha_t:=\alpha(t,X_t)$ is a Markovian solution to \ref{SP}.
\label{MainConv}
\end{theorem}

\begin{remark} Using standard parabolic PDE techniques and the regularity of $\phi$, we can show that, provided $m_0$ admits a density in $\cC_b^{2+\alpha}(\R^d)$, $m(t)$ in Theorem \ref{MainGeneral} admits a density $m(t,x)$ with respect to the Lebesgue measure such that $m \in \cC_b^{\frac{2+\alpha}{2},2+\alpha}([0,T] \times \R^d)$.
\end{remark} 

\begin{remark} In Theorems \ref{MainGeneral} and \ref{MainConv}, the stochastic basis $(\Omega, \mathcal{F}, \mathbb{F}, \mathbb{P})$ and the Brownian motion $(B_t)$ introduced at the beginning of this section are a priori fixed. In the terminology of stochastic control it means that we deal with strong solutions to the stochastic control problem.
\end{remark}

\begin{remark} In the spirit of the Karush-Kuhn-Tucker theorem, multiple inequality constraints $\displaystyle \Psi_i(m(T)) \leq 0$ $\forall i \in \llbracket 1,n \rrbracket $ can be considered provided they satisfy some qualification condition. We would say that the constraint is qualified at $\tilde{m} \in \mathcal{P}_1(\R^d)$ provided there exists some $m \in \mathcal{P}_1(\R^d)$ such that $\displaystyle \int_{\R^d} \frac{\delta \Psi_i}{\delta m}(\tilde{m},x)dm(x) <0$ for all $i \in \llbracket 1,n \rrbracket$ such that $\Psi_i(\tilde{m})= 0$. If $n=2$ a sufficient condition would be $\displaystyle \frac{\delta \Psi_i}{\delta m}(\tilde{m},.) \in L^2(\R^d)$ for $i=1,2$ and $\displaystyle \int_{\R^d} \frac{\delta \Psi_1}{\delta m}(\tilde{m},x) \frac{\delta \Psi_2}{\delta m}(\tilde{m},x) dx >0$. For $n\geq 2$ the condition would be satisfied everywhere if the constraints $\Psi_i$ are convex, satisfy Assumption \ref{FCg} and if there is some $m \in \mathcal{P}_1(\R^d)$ such that $\Psi_i(m) <0$ for all $i \in \llbracket 1,n \rrbracket$.
\end{remark}

\section{A relaxed problem: optimal control of the Fokker-Planck equation}

\label{sec: A relaxed Problem}

\begin{definition}
The relaxed problem is 
\begin{equation}
\label{RPg}
\tag{RP}
 \inf_{(m,\omega,W) \in \mathbb{K}} J_{RP}(m,\omega,W)
 \end{equation}
where $\mathbb{K}$ is the set of  triples $(m,\omega, W) \in \mathcal{C}^0([0,T],\mathcal{P}_1(\R^d)) \times \mathcal{M}([0,T]\times \R^d,\R^d) \times  \mathcal{M}([0,T]\times \R^d,\mathbb{S}_d(\R))$ such that $\omega$ and $W$ are absolutely continuous with respect to $m(t)\otimes dt$, 
\begin{equation}
 \partial_t m + \mdiv \omega -\sum_{i,j} \partial^2_{ij} W_{ij} = 0 
 \label{FPEncoreUne}
 \end{equation}
holds in the sense of distributions, $m(0)=m_0$ and $\Psi(m(T)) \leq 0$. The cost $J_{RP}$ is defined on $\mathbb{K}$ by

\begin{align*}
J_{RP}(m,\omega,W) &:= \int_0^T \int_{\R^d} L \left( t,x,\frac{d\omega}{dt \otimes dm}(t,x), \frac{dW}{dt \otimes dm}(t,x) \right) dm(t)(x)dt \\
& +\int_0^T f_2(t,m(t))dt +g(m(T)).
\end{align*}
\end{definition}

 Notice that the first term in the objective function $J_{RP}$ is convex in the variables $(m, \omega, W)$ and that the Fokker-Planck equation and the initial condition are linear in $(m , \omega, W )$. Therefore the problem is linear/convex when the final constraint as well as the costs $f_2$ and $g$ are convex.

\noindent We say that $(m,\omega, W)$ in $\mathcal{C}^0([0,T],\mathcal{P}_1(\R^d)) \times \mathcal{M}([0,T]\times \R^d,\R^d) \times  \mathcal{M}([0,T]\times \R^d,\mathbb{S}_d(\R))$ satisfies the Fokker-Planck equation (FPE) \ref{FPEncoreUne} with initital condition $m(0)=m_0$ if and only if, for all $\varphi \in \cC(\R^d)$ with compact support and all $\phi \in \mathcal{C}^{1,2} ((0,T) \times \R^d) $ with compact support we have

$$ \displaystyle \int_0^T \int_{\R^d} \partial_t \phi(t,x) dm(t)(x)dt +  \int_0^T \int_{\R^d} D \phi(t,x) . d\omega(t, x) + \int_0^T \int_{\R^d} D^2 \phi(t,x) . dW(t,x) = 0$$
 and the initial condition $\displaystyle \int_{\R^d} \varphi(x) dm(0)(x) = \int_{\R^d} \varphi(x) dm_0(x)$.

\noindent Moreover, if $\omega$ and $W$ are absolutely continuous with respect to $m(t) \otimes dt$ the above relations hold if $\varphi$ and $\phi$ are respectively taken in $\mathcal{C}_b(\R^d)$ and $\mathcal{C}^{1,2}_b((0,T) \times \R^d)$ (see \cite{Trevisan2016} Remark 2.3). In this case, we have for all $\phi \in \mathcal{C}_b^{1,2} ([0,T] \times \R^d)$ and for all $t_1,t_2 \in [0,T]$

\begin{align*}
\int_{\R^d} \phi(t_2,x)dm(t_2)(x) &= \int_{\R^d} \phi(t_1,x)dm(t_1)(x) \\
&+ \int_{t_1}^{t_2} \left[ \partial_t\phi(t,x) + D\phi(t,x). \frac{d\omega}{dm \otimes dt}(t,x) + \frac{dW}{dm \otimes dt}(t,x).D^2\phi(t,x) \right]dm(t)(x)dt. 
\end{align*}

\noindent Let us recall some known results about the link between solutions of the FPE and solutions to the SDE.

\begin{proposition} 
\label{Equivalence FPE/SDE}

\noindent
\begin{enumerate}

\item Suppose that $m$ is a solution to the Fokker-Planck equation
\begin{equation}
\left \{
\begin{array}{ll}
\partial_t m + \mdiv (b(t,x) m) - \sum_{i,j} \partial^2_{i,j} \left ( (\sigma^t\sigma(t,x) )_{ij} m \right) =0 \\
m(0) = m_0.
\end{array}
\right.
\label{FPE eq FPE/SDE}
\end{equation}
 with coefficients $b : [0,T] \times \R^d \rightarrow \R^d$, $\sigma : [0,T] \times \R^d \rightarrow \mathbb{M}_d(\R)$, Borel functions satisfying $$ \int_{0}^T \int_{\R^d} \left( |b(t,x)| + |\sigma(t,x)|^2 \right) dm(t)(x)dt <+\infty.$$
Then there is a filtered probability space $(\Omega, \mathcal{F}, \mathbb{F}, \mathbb{P} )$, an adapted Brownian motion $(B_t)_{t\geq 0}$ and an adapted process $(X_t)_{0 \leq t \leq T}$ such that
  
$$ \mathcal{L}(X_0 ) = m_0, $$
$$ dX_t = b(t,X_t)dt + \sqrt{2}\sigma(t,X_t) dB_t. $$
\noindent Moreover, for all $t \in [0,T]$, $\mathcal{L}(X_t) = m(t)$.

\item Conversely, suppose that $(X_s)_{s\geq 0}$ is a strong solution of the stochastic differential equation

\begin{equation*}
\left \{
\begin{array}{ll}
dX_s = b(s,X_s)ds + \sqrt{2}\sigma(s,X_s)dB_s \\
X|_{t=0} = X_0
\end{array}
\right.
\end{equation*}
on some filtered probability space $(\Omega, \mathcal{F}, \mathbb{F},\PP)$ endowed with an adapted Brownian motion $(B_t)$ with $ b : [0,T] \times \R^d \rightarrow \R^d$ and $\sigma : [0,T] \times \R^d \rightarrow \mathbb{M}_d(\R) $ Borel-measurable functions such that 

$$\displaystyle \PP  \left[ \int_0^T \left( |b(s,X_s)| + |\sigma (s,X_s)|^2 \right) ds < +\infty \right ] =1$$
and let $m(t) := \mathcal{L}(X_t) = X_t \# \mathbb{P}$, then $m$ satisfies the Fokker-Planck equation \ref{FPE eq FPE/SDE}.

\end{enumerate}

\end{proposition}

\begin{proof}
\noindent The second part follows from Itô's lemma and is standard. For the first part we need to combine the argument of \cite{Karatzas1991} and \cite{Trevisan2016}. From \cite{Trevisan2016} Theorem 2.5 we know that this statement is equivalent to the existence of a solution to the so-called martingale problem and from \cite{Karatzas1991} Chapter 4, we know that existence of a solution to the martingale problem is equivalent to the existence of a weak solution to the SDE.
\end{proof}

Let $V_{RP}(m_0)$ be the value of the relaxed problem. The link with the usual compactification / convexification (see \cite{ElKaroui1987} and \cite{Lacker2015}) method in stochastic optimal control is the following~:

\begin{proposition}

$$ V_{RP}(m_0) =  \inf_{\mathbf{q}_{A},m} \lbrace \int_0^T \int_{\R^d} \int_{A} f_1(t,x,a) d\mathbf{q}_{A}(t,x)(a) dm(t)(x)dt +\int_0^T f_2(t,m(t))dt+g(m(T)) \rbrace $$

\noindent where the infimum is taken over the couples $(\mathbf{q}_{A},m) \in L^0([0,T]\times \R^d,\mathcal{P}_1(A)) \times \mathcal{C}^0([0,T], \mathcal{P}_1(\R^d)) $ that satisfy in the sense of distributions the Fokker-Planck equation

$$ \partial_t m + \mdiv (\int_{A} b(t,x,a) d\mathbf{q}_{A}(t,x)(a) m) - \sum_{i,j} \partial^2_{ij} ( \left ( \int_{A}\sigma ^t\sigma(t,x,a) d\mathbf{q}_{A}(t,x)(a) \right)_{ij} m ) = 0 $$
together with the initial condition $m(0)=m_0$ and the terminal constraint $\Psi(m(T)) \leq 0$.

\end{proposition}

\begin{proof}

\noindent The proof follows from the dual representation of $L$ in Lemma \ref{dual representation for $H_1^*$ and $H_2^*$} and a measurable selection argument as in \cite{Stroock1997} Theorem 12.1.10. For every competitor $(m,\omega,W)$ such that $\displaystyle L(t,x, \frac{d\omega}{dt\otimes dm}(t,x),\frac{dW}{dt\otimes dm}(t,x)) < +\infty $  one can find a measurable function $\mathbf{q}_{A} : [0,T]\times \R^d \rightarrow \mathcal{P}_1(A)$ such that, for every $(t,x) \in [0,T]\times \R^d$ one has

$$  L(t,x,\frac{d\omega}{dt \otimes dm}(t,x), \frac{dW}{dt \otimes dt}(t,x)) = \int_{A} f_1(t,x,a) d\mathbf{q}_{A}(t,x)(a), $$
and
$$ \left( \frac{d\omega}{dt \otimes dm}(t,x),  \frac{dW}{dt \otimes dm}(t,x) \right)  = \left( \int_{A} b(t,x,a) d\mathbf{q}_{A}(t,x)(a), \int_{A} \sigma^t\sigma(t,x,a) d\mathbf{q}_{A}(t,x)(a) \right). $$

\end{proof}

\subsection{Analysis of the relaxed problem}

\noindent We will need the following facts :

\begin{lemma}
There exists $(m,\omega,W) \in \mathbb{K} $  such that $J_{RP}(m, \omega,W) < +\infty$.
\label{Reachability}
\end{lemma}

\begin{proof}

We have to check that we can indeed reach the final constraint with a finite cost. By continuity of $\Psi$ we can find $x_0,...,x_n \in \R^d$ such that $\Psi( \frac{1}{n}\sum_{i=1}^n \delta_{x_i} ) <0$. Fix some $\delta >0$. Let $i$ be in $\llbracket 0,n \rrbracket $. For all $c>0$ we can find $\mathbf{q}^c \in L^0([0,T] \times \R^d, \mathcal{P}_1(A))$ such that $\displaystyle \int_{A} b(t,x,a) d\mathbf{q}^c(t,x)(a) = c(x_i-x)$ and $\displaystyle     \int_A f_1(t,x,a) d\mathbf{q}^c(t,x)(a) \leq \alpha_1' c^{r_1^*}|x_i-x|^{r_1^*} +C_H$ for all $(t,x) \in [0,T] \times \R^d$ (see Lemma \ref{GrH_1 explained}).
We define the measurable function $\displaystyle \tilde{\sigma}_c (t,x) := \left( \int_A \sigma^t\sigma(t,x,a)d\mathbf{q}^c (t,x)(a) \right)^{\frac{1}{2}} $. Notice that $\Lambda^- I_d \leq \tilde{\sigma}_c^t\tilde{\sigma}_c(t,x) \leq \Lambda^+ I_d$ for all $(t,x) \in [0,T] \times \R^d$. We can use the result of Krylov in part 2.6 of \cite{Krylov1980} (existence of weak solutions to stochastic differential equations with bounded measurable coefficients and uniformly non-degenerate volatility) and find a filtered probability space $(\Omega^1, \mathcal{F}^1,\mathbb{F}^1, \mathbb{P}^1)$ satisfying the usual conditions, an adapted Brownian motion $(B_t)$, a $\mathcal{F}^1_0$ measurable random variable $X_0$ with law $m_0$ and a solution $Y_t^c$ of the stochastic differential equation

$$dY_t^c = ce^{ct}x_idt + \sqrt{2} \tilde{\sigma}_c'(t,Y_t^c)dB_t$$
starting from $X_0$ with $\tilde{\sigma}_c'(t,y) = e^{ct}\tilde{\sigma}_c(t,e^{-ct}y)$. By Ito's lemma, $X_t^c:= e^{-ct}Y_t^c$ solves the SDE 

$$dX^c_t = c(x_i - X_t^c)dt + \sqrt{2} \tilde{\sigma}_c(t,X_t^c)dB_t$$
starting from $X_0$ and we have, for all $t\in [0,T]$

$$X_t^c = x_i + (X_0-x_i)e^{-ct} + \sqrt{2}e^{-ct} \int_0^t \tilde{\sigma}_c(s,X_s^c) e^{cs}dB_s.$$

\noindent Using the Burkholder-Davis-Gundy inequality and the upper bound on $\sigma^t\sigma $ we get

\begin{align*}
\mathbb{E}^1(|X_t^c-x_i|^{r_1^*}) & \leq 2^{r_1^*-1} e^{-r_1^*ct} \mathbb{E}^1(|X_0-x_i|^{r_1^*}) + 2^{\frac{3r_1^*-2}{2}}e^{-r_1^*ct} \E^1 \left ( | \int_0^t \tilde{\sigma}_c(s,X_s^c)e^{cs} dB_s |^{r_1^*} \right) \\
& \leq 2^{r_1^*-1} e^{-r_1^*ct} \mathbb{E}^1(|X_0-x_i|^{r_1^*}) + 2^{\frac{3r_1^*-2}{2}}e^{-r_1^*ct} \E^1 \left ( ( \int_0^t \Tr(\tilde{\sigma}^t\tilde{\sigma}_c(s,X_s^c)e^{2cs} ds)^{\frac{r_1^*}{2}} \right) \\
&\leq 2^{r_1^*-1} e^{-r_1^*ct} \mathbb{E}^1(|X_0-x_i|^{r_1^*}) + 2^{\frac{3r_1^*-2}{2}} (d\Lambda^+)^{\frac{r_1^*}{2}} \left( \frac{e^{-2ct}-1}{2c} \right)^{\frac{r_1^*}{2}}
\end{align*}
where $\E^1$ is the expectation under $\mathbb{P}^1$. In particular, taking $t=T$ we see that, for $c$ sufficiently large we have $d_{r_1^*}(\mathcal{L}(X_T^c), \delta_{x_i}) \leq \delta$. Now, for such a $c$, we let $m^{i}(t) = \mathcal{L}(X_t^c)$, $\displaystyle \omega^{i} =c(x_i-x)m^{i}$, $W^{i} = \tilde{\sigma}_c^t\tilde{\sigma}_c(t,x) m^{i}$ . Since $f_2$ and $g$ are bounded functions, and thanks to the upper bound on $f_1$ we have that 

$$J_{RP}( m^{i},\omega^{i},W^{i} ) \leq C\left( 1+ \int_0^T \mathbb{E}^1(|X_t^c-x_i|^{r_1^*}) dt \right) < +\infty. $$
Now we do the same for all $i \in \llbracket 0, n \rrbracket$ and we let $\displaystyle (m,\omega,W) := \frac{1}{n} \sum_{i=1}^n (m^{i},\omega^{i},W^{i})$. The triple $(m,\omega,W)$ solves the Fokker-Planck equation starting from $m_0$. Now by convexity of 

$$\displaystyle (m,\omega,W) \rightarrow \int_0^T \int_{\R^d} L \left( t,x,\frac{d\omega}{dm \otimes dt }(t,x), \frac{dW}{dm \otimes dt}(t,x) \right) dm(t)(x)dt $$ 
and using the fact that $f_2$ and $g$ are bounded we get that $\displaystyle J_{RP}(m,\omega,W) <+\infty$. Finally
$$ d_{r_1^*} \left( \frac{1}{n} \sum_{i=1}^n m^{i}(T) , \frac{1}{n} \sum_{i=1}^n \delta_{x_i} \right) \leq C(n)\delta $$
for some non negative constant $C(n)$. For $\delta$ small enough we get that $J_{RP}(m,\omega,W) <+\infty$ and $\Psi(m(T)) <0$ which concludes the proof.

\end{proof}

\begin{lemma}
\label{AprioriEstimates}

\begin{enumerate}
\item Any point $(m,\omega,W) \in \mathbb{K}$ with $J_{RP}(m,\omega,W) < +\infty$, satisfies the following estimate for some constant $C_{r_2}$ depending only on $r_2$: for any $0<s \leq t<T$,
 
\begin{equation}
 \mathbf{d}_{r_2^*} (m(s),m(t))^{r_2^*} \leq C_{r_2} (t-s)^{r_2^*-1} \int_{\R^d} \int_0^T \left| \frac{d\omega}{dt \otimes dm}(u,x) \right|^{r_2^*} dm(u)(x)du + C_{r_2} \Lambda^+ (t-s)^{\frac{r_2^*}{2}}.
\label{Estimate1}
\end{equation}

\item There exists some $M>0$ such that 
\begin{equation}
\displaystyle \sup_{t \in [0,T]} \int_{\R^d} |x|^{r_2^*} dm(t)(x) + |\omega|([0,T] \times \R^d) + |W|([0,T] \times \R^d)  \leq M \label{Estimate2}
\end{equation}
whenever $J_{RP}(m,\omega,W) \leq \inf J_{RP} +1.$

\end{enumerate}

\end{lemma}

\begin{proof}

\noindent First observe that, since $J_{RP}(m,\omega,W) < +\infty$, by the dual formula for $L$ of Lemma \ref{dual representation for $H_1^*$ and $H_2^*$}, we know that $m(t) \otimes dt$-almost-everywhere~: $\displaystyle  \Lambda^- I_d \leq \frac{dW}{dt\otimes dm} \leq \Lambda^+ I_d$. Let $(\Omega, \mathcal{F}, \mathbb{F}, \mathbb{P}, (X,B))$ be a weak solution to the SDE 

$$dX_t = \frac{d\omega}{dt\otimes dm}(t,X_t)dt + \sqrt{2\frac{dW}{dt\otimes dm}}(t,X_t)dB_t $$
with $\mathcal{L}(X_t) = m(t)$ for all $t \in [0,T]$. The existence of such a solution is ensured by the fact that $m$ solves the FPE with coefficients  $\displaystyle \frac{d\omega}{dt \otimes dm}, \frac{dW}{dt \otimes dm}$ (see Proposition \ref{Equivalence FPE/SDE}). Now, for all $0 \leq s < t \leq T$, with $M_{r_2}$ and $C_{r_2}$ positive constants depending only on $r_2$ we have

\begin{align*}
\mathbf{d}_{r_2^*} (m(s),m(t)) ^{r_2^*}  & \leq \E \left[ |X_t - X_s |^{r_2^*} \right] \\
& \leq 2^{r_2^* -1} \E \left[ |\int_s^t \frac{d\omega}{dt \otimes dm}(u,X_u)du |^{r_2^*} \right] + 2^{r_2^* -1} \E \left[ |\int_s^t \sqrt{2\frac{dW}{dt \otimes dm}}(u,X_u)dB_u |^{r_2^*} \right] \\
& \leq (2(t-s))^{r_2^* -1} \E \left[ \int_s^t \left| \frac{d\omega}{dt \otimes dm}(u,X_u) \right|^{r_2^*} du \right] \\
&+ 2^{r_2^*} M_{r_2} \E \left( \left[ \int_s^t \Tr(\frac{dW}{dt \otimes dm})(u,X_u)du \right]^{\frac{r_2^*}{2}} \right) \\
& \leq C_{r_2} (t-s)^{r_2^*-1} \int_{\R^d} \int_0^T \left| \frac{d\omega}{dt \otimes dm}(u,x) \right|^{r_2^*} dm(u)(x)du + C_{r_2} \Lambda^+ (t-s)^{\frac{r_2^*}{2}},
\end{align*}
where we used Jensen inequality for the term involving $\omega$ and Burkholder-Davis-Gundy inequality for the other one.

For the second part of the lemma, let us take $(m,\omega,W) \in \mathbb{K}$ such that $J_{RP}(m,\omega,W) \leq \inf J_{RP} + 1$. From the growth assumptions on $L$, there exists $M_1>0$ (which does not depend on the particular $(m,\omega,W)$) such that $$ \int_{\R^d} \int_0^T \left| \frac{d\omega}{dt \otimes dm}(u,x) \right|^{r_2^*} dm(u)(x)du \leq M_1.$$ 
\noindent Using the estimate proven in the first part of the lemma, we see that for all $t,s \in [0,T]$,  $\mathbf{d}_{r_2^*}(m(s),m(t)) \leq M_1'$ for some $M_1' >0$ which, once again, does not depend on the particular choice of $(m,\omega,W)$. This yields the uniform estimate on $\displaystyle \int_{\R^d} |x|^{r_2^*} dm(t)(x) < M_1'' $ for some new $M_1'' >0$. The uniform estimate on $|\omega|$ follows by Hölder's inequality
\begin{align*}
|\omega|([0,T]\times \R^d) &\leq \left ( \int_0^T \int_{\R^d} dm(t)(x)dt \right )^{1/r_2} \left ( \int_0^T \int_{\R^d} \left |\frac{d\omega}{dt \otimes dm}(t,x) \right |^{r_2^*} dm(t)(x)dt \right )^{1/r_2^*} \\
&\leq T^{1/r_2}M_1^{1/r_2^*}.
\end{align*}

\noindent Finally, $m(t)\otimes dt$-almost everywhere $ \displaystyle \Sp \left ( \frac{dW}{dt\otimes dm}(t,x) \right )  \in [\Lambda^-,\Lambda^+ ]$ which means that $|W|([0,T] \times \R^d) \leq \sqrt{d}\Lambda^+ $. The claim follows taking $M = M_1'' + T^{1/r_2}M_1^{1/r_2^*} + \sqrt{d}\Lambda^+$.

\end{proof}

\noindent From this we can conclude with:

\begin{theorem}
\label{Existence for RP}
$J_{RP}$ achieves its minimum at some point $(\tilde{m},\tilde{\omega},\bar{W})$ in $\mathbb{K}$.
\end{theorem}

\begin{proof}

This follows from the direct method of Calculus of Variations. Let $(m_n,\omega_n,W_n)$ be a minimizing sequence such that, for all $n \in \mathbb{N}$, $ J_{RP}(m_n,\omega_n,W_n) \leq \inf J_{RP}+1$. Using the Estimate \ref{Estimate2} in Lemma \ref{AprioriEstimates} we can use Arzela-Ascoli theorem on the one hand and Banach-Alaoglu theorem on the other hand to extract a subsequence (still denoted $(m_n,\omega_n,W_n)$) converging to $(\tilde{m}, \tilde{\omega}, \tilde{W}) \in \mathcal{C}^0([0,T],\mathcal{P}_{r_2^*}(\R^d)) \times \mathcal{M}([0,T]\times \R^d,\R^d) \times  \mathcal{M}([0,T]\times \R^d,\mathbb{S}_d(\R))$ in $\mathcal{C}^0([0,T],\mathcal{P}_{\delta}(\R^d)) \times \mathcal{M}([0,T]\times \R^d,\R^d) \times  \mathcal{M}([0,T]\times \R^d,\mathbb{S}_d(\R))$ for any $\delta \in (1,r_2^*)$. It remains to show that $(\tilde{m}, \tilde{\omega}, \tilde{W})$ belongs to $\mathbb{K}$ and is indeed a minimum. The Fokker-Planck equation and the initial and final conditions are easily deduced from the weak-$*$ convergence of measures. To conclude we can use Theorem 2.34 of \cite{Ambrosio2000} to show that absolute continuity of $\omega_n$ and $W_n$ with respect to $m_n(t) \otimes dt$ is preserved when we take limits and that $\displaystyle J_{RP}(\tilde{m},\tilde{\omega},\tilde{W}) \leq \liminf_{n} J_{RP}(m_n,\omega_n,W_n)$. So $(\tilde{m},\tilde{\omega},\tilde{W})$ is indeed a minimum of $J_{RP}$ in $\mathbb{K}$.
\end{proof}

\subsection{Necessary conditions for the linear case}

\label{sec: Necessary Conditions}

In this section we suppose that $\Psi$ is linear: there is a function $h : \R^d \rightarrow \R$ such that, for all $m \in \cP_1(\R^d)$, $\displaystyle \Psi(m) = \int_{\R^d} h(x)m(dx) $. We also suppose that $h$ belongs to $\cC_b^{3+\alpha}(\R^d)$ for some $\alpha \in (0,1)$ and that there exists $x_T \in \R^d$ such that $h(x_T) < 0$. Under these assumptions, $\Psi$ satisfies Assumption \ref{FCg}. We also suppose that $f_2$ and $g$ are linear in $m$ with $\displaystyle f_2(t,m) = \int_{\R^d}f'_2(t,x)dm(x)$ 	and $\displaystyle g(m) = \int_{\R^d} g'(x)dm(x)$ with $g' \in \mathcal{C}^{3+\alpha}_b(\R^d)$ and $f_2'$ satisfying the assumptions of Theorem \ref{MainHJB}. Let us introduce a dual problem for \ref{RPg}.

\begin{definition}[Dual Problem]
\label{RP*}

\noindent The dual problem is :
\begin{equation}
\sup_{(\lambda,\phi) \in \R^+ \times \mathbb{A}, \phi \in \mbox{HJ}^-(\lambda h +g) } \int_{\R^d} \phi(0,x) m_0(dx) 
\tag{DP}
\label{DP}
\end{equation}

\end{definition}

\noindent where $ \mathbb{A} = \cC^{1,2}_b([0,T]\times \R^d) $ and, for all $(\lambda, \phi) \in \R^+ \times \mathbb{A} $, $\phi$ belongs to $\mbox{HJ}^-(\lambda h +g')$ if and only if :

\begin{equation}
\left \{
\begin{array}{ll}
-\partial_t \phi(t,x) + H(t,x,D\phi(t,x),D^2 \phi(t,x)) \leq f'_2(t,x) \mbox{ in     } [0,T]\times \R^d  \\
\phi(T,x) \leq \lambda h(x) + g'(x) \mbox{ in   } \R^d
\end{array}
\right.
\end{equation}

The main theorem of this part is a duality result between \ref{RPg} and \ref{DP} :

\begin{theorem}
$$\min_{(m,\omega,W) \in \mathbb{K}} J_{RP}(m,\omega,W) = \sup_{(\lambda,\phi) \in \R^+ \times \mathbb{A}, \phi \in \mbox{HJ}^-(\lambda h +g) } \int_{\R^d} \phi(0,x) m_0(dx). $$
\label{Duality}
\end{theorem}

To prove Theorem \ref{Duality} the idea is to write the relaxed problem \ref{RPg} as a min/max problem and use the Von Neumann theorem to conclude. The statement of the Von-Neumann theorem is given in Appendix \ref{sec: Von-Neumann Theorem and Proof of the dual representation for $H_1^*$ and $H_2^*$}.

\begin{proof}[Proof of Theorem \ref{Duality}]

\textbf{Step 1: Further Relaxation}

First we need to enlarge the space of test functions $\mathbb{A}$ to allow for functions with linear growth. More precisely, we define $\mathbb{A}'$ as the subset of  $\mathcal{C}^{1,2} ([0,T] \times \R^d)$ consisting of functions $\phi$ such that

$$ \|(\partial_t \phi)^- \|_{\infty} + \| \phi^+ \|_{\infty} + \|D\phi\|_{\infty} + \|D^2 \phi \|_{\infty} + \left \| \frac{|\phi| + |\partial_t \phi |}{1+ |x|} \right \|_{\infty} < +\infty. $$

Owing to the estimates of Lemma \ref{AprioriEstimates} and using an approximation argument similar to \cite{Trevisan2016} Remark 2.3 we see that any minimizer of the relaxed problem \ref{RPg} satisfies the Fokker-Planck equation against any function $\phi \in \mathbb{A}'$. Now we define $\mathbb{B}$ to be the set of tuples $(m,\omega,W,n)$ in $\mathcal{M}^+([0,T]\times \R^d) \times \mathcal{M}([0,T]\times \R^d,\R^d) \times  \mathcal{M}([0,T]\times \R^d,\mathbb{S}_d(\R)) \times \mathcal{M}^+(\R^d) $ such that $\omega$ and $W$ are absolutely continuous with respect to $m$. $J_{RP}'$ is defined on $\mathbb{B}$ by

$$J_{RP}'(m,\omega,W,n) = \int_0^T \int_{\R^d} \left[ L \left(t, x,\frac{d\omega}{dm}(t,x), \frac{dW}{dm}(t,x) \right) + f'_2(t,x) \right] dm(t,x) + \int_{\R^d} g'(x) dn(x).  $$
If $(\tilde{m},\tilde{\omega},\tilde{W})$ is a solution of the relaxed problem we claim that

$$ J_{RP}(\tilde{m},\tilde{\omega},\tilde{W}) = J_{RP}'(\tilde{m},\tilde{\omega},\tilde{W},\tilde{m}(T,dx)) = \inf_{(m,\omega,W,n)} J_{RP}'(m,\omega,W,n), $$
where the infimum is taken over the $(m,\omega,W,n)$ in $\mathbb{B}$ satisfying, 

\begin{equation}
\displaystyle \forall \phi \in \mathbb{A}' \mbox{,     } \int_0^T \int_{\R^d} (\partial_t \phi m +  D\phi.\omega +  D^2\phi . W)  + \int_{\R^d} \phi(0,x)dm_0(x) - \int_{\R^d} \phi(T,x)dn(x) = 0,
\label{FP3} 
\end{equation}

\begin{equation}
\displaystyle \int_{\R^d} h(x) dn(x) \leq 0.
\label{ConstraintOnn}
\end{equation}

 Indeed, since $(\tilde{m},\tilde{\omega},\tilde{W},\tilde{m}(T)) $ belongs to $\mathbb{B}$ and satisfies the Fokker-Planck equation, it is clear that $ J_{RP}'(\tilde{m},\tilde{\omega},\tilde{W},\tilde{m}(T)) \geq \inf_{(m,\omega,W,n)} J_{RP}'(m,\omega,W,n)$. Now let us take $(m,\omega,W,n) \in \mathbb{B}$ satisfying \eqref{FP3} for every $\phi \in \mathbb{A}'$ and such that $J_{RP}'(m',\omega',W',n') < +\infty$. Testing \eqref{FP3} against space-independent functions we see that the time marginal of $m$ is the Lebesgue measure on $[0,T]$ and that $\displaystyle \int_{\R^d} dm(t)(x) =1$ $dt$-almost everywhere in $[0,T]$ for any flow of measures $t \rightarrow m(t)$ arising from the disintegration of $m$ with respect to $dt$.  Now we can follow Lemma \ref{Estimate1} and the discussion below in the proof of Theorem \ref{Existence for RP} to deduce that $m$ admits a continuous representative $m' \in \mathcal{C}([0,T], \mathcal{P}_{r_2^*}(\R^d))$. We then get  $n = m'(T) $ from \eqref{FP3}. Therefore $(m',\omega,W)$ belongs to $\mathbb{K} $, $$J_{RP}'(m,\omega,W,n) = J_{RP}(m',\omega,W) \geq J_{RP}(\tilde{m},\tilde{\omega},\tilde{W})$$ and the claim is proved. Now, observe that, for any point $(m,\omega,W,n)$ in $\mathbb{B}$ 

\begin{align*}
 \sup_{ (\lambda, \phi) \in \R^+ \times \mathbb{A'}} &\left[ \int_0^T \int_{\R^d} (\partial_t \phi m  + D\phi.\omega + D^2\phi . W)  + \int_{\R^d} \phi(0,x)dm_0(x) + \int_{\R^d} \left( \lambda h(x) - \phi(T,x) \right) dn(x) \right] \\
&= \left\{ \begin{array}{ll} 0 &\mbox{   if   }(m,\omega,W,n) \mbox{  satisfies \ref{FP3} and \ref{ConstraintOnn},   } \\ +\infty &\mbox{  otherwise. } \end{array} \right.
\end{align*}

\noindent Therefore we deduce that

\begin{align*}
V_{RP}(m_0) := \min_{(m,\omega,W) \in \mathbb{K} } J_{RP}(m,\omega,W) &= \inf_{(m,\omega,W,n) \in \mathbb{B} } \sup_{(\lambda , \phi) \in \R^+ \times \mathbb{A'}}  \mathcal{L}( (\lambda , \phi ), (m,\omega,W,n) ), \\
\end{align*}

\noindent where $\mathcal{L} : \R^+ \times \mathbb{A'}\times \mathbb{B} \rightarrow \R $ is defined by:

\begin{align*}
\displaystyle \mathcal{L}((\lambda,\phi),(m,\omega,W,n)) = &\int_0^T \int_{\R^d} (L \left(t, x,\frac{d\omega}{dm}(t,x), \frac{dW}{dm}(t,x) \right) + f'_2(t,x) ) dm(t,x) + \int_{\R^d} g'(x) dn(x) \\
\displaystyle &+ \int_0^T \int_{\R^d} \partial_t \phi(t,x) dm(t,x) + D\phi(t,x).d\omega(t,x) + D^2\phi(t,x) . dW(t,x) \\
\displaystyle &+ \int_{\R^d} \phi(0,x)dm_0(x) - \int_{\R^d} \phi(T,x)dn(x)  + \lambda \int_{\R^d} h(x) dn(x). 
\end{align*}

\noindent \textbf{Step 2: Analysis of the Lagrangian}  

We immediately check that for all $(\lambda, \phi) \in \R^+ \times \mathbb{A}'$, $(m,\omega,W,n) \rightarrow \mathcal{L}((\lambda , \phi ), (m,\omega,W,n) )$ is convex and for all $(m,\omega,W,n) \in \mathbb{B}$, $(\lambda,\phi) \rightarrow \mathcal{L}((\lambda , \phi ), (m,\omega,W,n) )$ is concave. Now $\mathcal{L}$ can be rewritten as the sum of four terms, $\displaystyle \mathcal{L} = \mathcal{L}_1 +\mathcal{L}_2 +\mathcal{L}_3 + \int_{\R^d} \phi(0,x)m_0(dx)$ where, 

\begin{align*}
\displaystyle &\mathcal{L}_1 ((\lambda,\phi),m) := \int_0^T \int_{\R^d} \left[ \partial_t \phi(t,x) - H(t,x,D\phi(t,x),D^2\phi(t,x)) +f'_2(t,x)\right] dm(t,x), \\
\displaystyle & \mathcal{L}_2((\lambda,\phi),(m,\omega,W )) := \int_0^T \int_{\R^d} f^{(\lambda,\phi)} \left(t,x,\frac{d\omega}{dm}(t,x),\frac{dW}{dm}(t,x) \right) dm(t,x), \\
\displaystyle & \mathcal{L}_3((\lambda,\phi),n)  := \int_{\R^d} \left[ \lambda h(x) + g'(x) - \phi(T,x) \right] dn(x),
\end{align*}

\noindent with $f^{(\lambda,\phi)}:[0,T] \times \R^d \times \R^d \times \mathbb{S}_d(\R) \rightarrow \R$ defined by, 
$$f^{(\lambda,\phi)} (t,x,p,N) =  L\left( t, x, q ,N \right) + H(t,x,D\phi(t,x),D^2\phi(t,x)) + p.D\phi(t,x) + N.D^2 \phi(t,x).$$

\noindent Now suppose that $(m_k,\omega_k, W_k,n_k)_{k \in \mathbb{N}}$ weakly-$*$ converges to some point $(m,\omega,W,n)$ and satisfies the uniform estimate 

\begin{align}
\notag \max \lbrace & \int_0^T\int_{\R^d} (1+ |x|) dm_k(t,x) ,\int_{\R^d} (1+ |x|)dn_k(x),\int_0^T \int_{\R^d} \left | \frac{d\omega_k}{dm_k}(t,x) \right |^{r_2^*} dm_k(t,x), \\
&\int_0^T \int_{\R^d} \chi_{[0,\Lambda^+]} \left( \frac{dW_k}{dm_k}(t,x) \right) dm_k(t,x) \rbrace \leq M
\label{EstimateEncoreEtToujours}
\end{align}
for some $M>0$ and for all $k \in \mathbb{N}$. Then, owing to the fact that the integrands in $\mathcal{L}_1$ and $\mathcal{L}_3$ are bounded from below, we have that, for every $(\lambda, \phi) \in \R^d \times \mathbb{A}'$, $\displaystyle \mathcal{L}_1((\lambda,\phi),m) \leq \liminf_{k\rightarrow + \infty} \mathcal{L}_1((\lambda,\phi), m_k)$ and $\displaystyle \mathcal{L}_3((\lambda,\phi) , n) \leq \liminf_{k\rightarrow + \infty}\mathcal{L}_3((\lambda,\phi) , n_k)$. Moreover, $f^{(\lambda,\phi)}$ is nonnegative, lower-semicontinuous and for all $(t,x) \in [0,T] \times \R^d$, $(q,N) \rightarrow f^{(\lambda,\phi)}(t,x,q,N) $ is convex so we can proceed as in \cite{Ambrosio2000} Theorem 2.34 and Example 2.36 to prove that $\omega,W$ are absolutely continuous with respect to $m$ and $\displaystyle \mathcal{L}_2( ( \lambda,\phi),( m,\omega,W) ) \leq \liminf_{k \rightarrow + \infty} \mathcal{L}_2( ( \lambda,\phi), (m_k,\omega_k,W_k) )$. Finally, we have that

\begin{equation}
 \mathcal{L}( (\lambda, \phi), (m, \omega,W,n)) \leq \limsup_{k\rightarrow + \infty}  \mathcal{L}( (\lambda, \phi), (m_k, \omega_k,W_k,n_k)).
 \label{w-convergence}
\end{equation}

\noindent \textbf{Step 3: Min/Max argument}

 Now we are going to use the Von Neumann Theorem \ref{VN}  to show that 

$$ \inf_{(m,\omega,W,n) \in \mathbb{B}} \sup_{(\lambda,\phi) \in \R^+ \times \mathbb{A}'} \mathcal{L}((\phi,\lambda),(m,\omega,W,n)) = \sup_{(\lambda,\phi) \in \R^+ \times \mathbb{A}'} \inf_{(m,\omega,W,n) \in \mathbb{B}} \mathcal{L}((\lambda,\phi),(m,\omega,W,n)). $$
To check that the hypothesis of the theorem are satisfied, we define $\varphi^*(t,x) :=  \sqrt{1+ |x|^2 }  (t-T-1) $ and $\phi^*(t,x):= \left(\sqrt{1+ |x|^2 } +C_1 \right ) (t-T-1) + C_2$ where 
$$C_1 = \|H \left(., .,D\varphi^*(.,.),D^2 \varphi^* (.,.) \right) -f'_2(.,.)\|_{\infty} +1 $$ 
and $C_2 = - ||g'||_{\infty} - C_1 -1$. Then we let 
$$C^* := \sup_{(\lambda,\phi) \in \R^+ \times \mathbb{A}'} \inf_{(m,\omega,W,n) \in \mathbb{B} } \mathcal{L}(\lambda,\phi),(m,\omega,W,n)) +1 $$ 
and we check that 
$$ \mathbb{B}^* := \left \{ (m,n,\omega,W) \in \mathbb{B} \mbox{   such that   } \mathcal{L}((0,\phi^*),(m,\omega,W,n) ) \leq C^* \right \} $$
is not empty and that there exists some $M>0$ such that any $(m,\omega,W,n) \in \mathbb{B}^*$ satisfies Estimate \ref{EstimateEncoreEtToujours}. We deduce that $\mathbb{B}^*$ is (strongly) bounded and using \ref{w-convergence} we see that $\mathbb{B}^*$ is weakly-$*$ compact. Now we can use \ref{EstimateEncoreEtToujours} and \ref{w-convergence} once again to show that for all $C>0$ and all $(\lambda,\phi) \in \R^+ \times \mathbb{A}'$, 
\begin{equation*}
 \mathbb{B}^* \cap \left \{ (m,n,\omega,W) \in \mathbb{B} \mbox{   such that   } \mathcal{L}((\lambda,\phi),(m,\omega,W,n) ) \leq C \right \} 
 \end{equation*}
is (possibly empty and) compact. Therefore we can apply the Von Neumann theorem, Theorem \ref{VN} to show that  
$$ \inf_{(m,\omega,W,n) \in \mathbb{B}} \sup_{(\lambda,\phi) \in \R^+ \times \mathbb{A}'} \mathcal{L}((\phi,\lambda),(m,\omega,W,n)) = \sup_{(\lambda,\phi) \in \R^+ \times \mathbb{A}'} \inf_{(m,\omega,W,n) \in \mathbb{B}} \mathcal{L}((\lambda,\phi),(m,\omega,W,n)). $$

\noindent \textbf{Step 4: Computation of the dual problem}

 Let $(\lambda, \phi) \in \R^+ \times \mathbb{A}'$ be fixed and consider the problem 
$$ \inf_{(m,\omega,W,n) \in \mathbb{B}} \mathcal{L}((\lambda,\phi),(m,\omega,W,n)). $$
Recall the definitions of $\mathcal{L}_1$, $\mathcal{L}_2$ and $\mathcal{L}_3$ in Step $2$ of the proof and observe first that, for fixed $(m,n)$,
$$ \inf_{(\omega,W)} \mathcal{L}_2 ((\lambda,\phi), (m, \omega,W)) = 0 $$
with the infimum being achieved if and only if,
\begin{equation*}
\left \{
\begin{array}{ll}
\omega = -\partial_p H(t,x,D\phi(t,x),D^2\phi(t,x)) m, \\
 W = -\partial_M H(t,x,D\phi(t,x),D^2\phi(t,x))  m. 
 \end{array}
 \right.
 \end{equation*}
Therefore it holds that
 
\begin{align*}
\inf_{(m,\omega,W,n) \in \mathbb{B}} \mathcal{L}((\lambda,\phi),(m,\omega,W,n)) &= \inf_{m \in \mathcal{M}^+([0,T] \times \R^d)} \mathcal{L}_1 ((\lambda,\phi),m) + \inf_{n \in \mathcal{M}^+(\R^d)} \mathcal{L}_3 ((\lambda,\phi),n) \\
&+ \int_{\R^d} \phi(0,x)dm_0(x) 
\end{align*}
but we have
 
$$ \inf_{m \in \mathcal{M}^+([0,T] \times \R^d)} \mathcal{L}_1 ((\lambda,\phi),m) = 
\left\{ \begin{array}{ll} 0 & \mbox{   if } -\partial_t \phi + H(t,x,D\phi,D^2 \phi)  \leq f'_2(t,x)  \mbox{  in } [0,T] \times \R^d, \\ -\infty & \mbox{  otherwise} \end{array} \right.
$$
and 
$$\inf_{n \in \mathcal{M}^+(\R^d)} \mathcal{L}_3 ((\lambda,\phi),n)  = 
\left\{ \begin{array}{ll} 0 & \mbox{   if } \phi(T,x) \leq \lambda h(x) + g'(x)  \mbox{  in } \times \R^d, \\ -\infty & \mbox{  otherwise.} \end{array} \right. $$
so we can conclude that

\begin{align*}
\inf_{(m,\omega,W,n) \in \mathbb{B}} \sup_{ (\lambda, \phi) \in \R^+ \times \mathbb{A}' } \mathcal{L}( (m,\omega,W,n),(\lambda, \phi)) &= \sup_{ (\lambda, \phi) \in \R^+ \times \mathbb{A'} } \inf_{(m,\omega,W,n) \in \mathbb{B}}  \mathcal{L}( (m,\omega,W,n),(\lambda, \phi)) \\
&= \sup_{ (\lambda, \phi) \in \R^+ \times \mathbb{A}', \phi \in \mbox{HJ}^-(\lambda h + g')} \int_{\R^d} \phi(0,x)m_0(dx),
\end{align*}

\noindent where $\phi \in \mbox{HJ}^-(\lambda h +g')$ for some $(\lambda, \phi) \in \R^+ \times  \mathbb{A}'$ if and only if

\begin{equation*}
\left \{
\begin{array}{ll}
-\partial_t \phi(t,x) + H(t,x,D\phi(t,x),D^2 \phi(t,x)) \leq f'_2(t,x) \mbox{ in     } [0,T]\times \R^d,  \\
\phi(T,x) \leq \lambda h(x) + g'(x) \mbox{ in   } \R^d.
\end{array}
\right.
\end{equation*}

\noindent Finally, we get $\displaystyle \min_{(m,\omega,W) \in \mathbb{K}} J_{RP}(m,\omega,W) = \sup_{(\lambda,\phi) \in \R^+ \times \mathbb{A}', \phi \in \mbox{HJ}^-(\lambda h +g') } \int_{\R^d} \phi(0,x) m_0(dx)$.
\end{proof}

Notice that this duality is not surprising and holds under very general conditions (see for instance \cite{Fleming1989}). In particular the volatility $\sigma$ can be degenerate. However the existence of solutions to the dual problem requires stronger assumptions. In particular we need strong solutions to the HJB equation and that is why we need Theorem \ref{MainHJB}.

\begin{lemma}
The dual problem has a finite value which is achieved at some point $(\tilde{\lambda},\tilde{\phi}) \in \R^+ \times \mathcal{C}_b^{1,2}([0,T] \times \R^d)$ such that :

\begin{equation*}
\left \{
\begin{array}{ll}
\displaystyle -\partial_t \tilde{\phi} + H(t,x,D\tilde{\phi}(t,x),D^2 \tilde{\phi}(t,x)) = f'_2(t,x) \mbox{ in     } [0,T]\times \R^d  \\
\displaystyle \tilde{\phi}(T,x) = \tilde{\lambda} h(x) + g'(x) \mbox{ in   } \R^d.
\end{array}
\right.
\end{equation*}

\label{lemmaHJB}
\end{lemma}

\begin{proof}

The finiteness follows from the fact that 
$$ \sup_{(\lambda,\phi) \in \R^+ \times \mathbb{A}', \phi \in \mbox{HJ}^-(\lambda h +g) } \int_{\R^d} \phi(0,x) dm_0(x) = \min_{(m,\omega,W) \in \mathbb{K}} J_{RP}(m,\omega,W) < +\infty.$$
Now take $(\bar{m},\bar{\omega},\bar{W}) \in \mathbb{K}$ such that $J_{RP}(\bar{m},\bar{\omega},\bar{W}) <+\infty$ and $\displaystyle \int_{\R^d} h(x)d \bar{m}(T)(x) < 0$ and $(\lambda, \phi)$ a candidate for the dual problem. Since $(\bar{m},\bar{\omega},\bar{W})$ satisfies the Fokker-Planck equation we have, taking $\phi$ as a test function

\begin{align*}
& \int_{\R^d} \phi(T,x)d\bar{m}(T)(x) = \int_{\R^d} \phi(0,x)dm_0(x) \\
&+ \int_0^T \int_{\R^d} \left[ \partial_t \phi(t,x) + \frac{ d\bar{\omega}}{d \bar{m} \otimes dt}(t,x).D\phi(t,x) + \frac{d\bar{W}}{d \bar{m} \otimes dt}(t,x).D^2\phi(t,x) \right]d\bar{m}(t)(x)dt. 
\end{align*}
Using the inequations satisfied by $\phi$ and the definition of $L$ we get after reorganizing the terms

\begin{equation}
\lambda \left( - \int_{\R^d} h(x)d\bar{m}(T)(x) \right) \leq J_{RP}(\bar{m}, \bar{\omega}, \bar{W}) - \int_{\R^d} \phi(0,x) dm_0(x).
\label{Estimatelambda2020}
\end{equation}
Now if we take $(\phi_n, \lambda_n)$ a maximizing sequence, the above inequality shows that $(\lambda_n)$ is bounded. Taking a subsequence we can suppose that $(\lambda_n)$ converges to some $\tilde{\lambda} \geq 0$. By comparison, $(\tilde{\phi},\tilde{\lambda})$ is a solution of the dual problem where $\tilde{\phi} \in \mathcal{C}_b^{1,2}([0,T] \times \R^d)$ is solution to 

$$
\left \{
\begin{array}{ll}
-\partial_t \phi(t,x) + H(t,x,D\phi(t,x),D^2 \phi(t,x)) = f_2'(t,x) \mbox{ in     } [0,T]\times \R^d  \\
\phi(T,x) = \tilde{\lambda} h(x) + g'(x) \mbox{ in   } \R^d.
\end{array}
\right.
$$  

\end{proof}

\begin{remark}
In the proof of the previous lemma, we showed as a by product that $\lambda$ is bounded independently from $\phi, m$. In particular using  inequality \ref{Estimatelambda2020} for a maximizing sequence and using the duality result of Theorem \ref{Duality} we get that $\tilde{\lambda}$ satisfies
\begin{equation*}
\displaystyle \tilde{\lambda} \leq \frac{ J_{RP}(\bar{m}, \bar{\omega}, \bar{W}) - V_{RP}(m_0) }{ \displaystyle - \int_{\R^d} h(x)d\bar{m}(T)(x) }
\end{equation*}
for any candidate $(\bar{m},\bar{\omega},\bar{W})$ such that $\displaystyle \int_{\R^d} h(x) d\bar{m}(T)(x) < 0$.
\end{remark}

\begin{corollary}
\label{OptimalityConditionsLinearConstraint}
If $(\tilde{m},\tilde{\omega},\tilde{W})$ and $(\tilde{\lambda},\tilde{\phi})$ are points where respectively the primal and the dual problems are achieved, then
$$ \tilde{\omega} = -\partial_p H(t,x,D\tilde{\phi}, D^2\tilde{\phi}(t,x)) \tilde{m}(t) \otimes dt,$$
$$\tilde{W} =-\partial_M H(t,x,D\tilde{\phi}, D^2\tilde{\phi}(t,x)) \tilde{m}(t) \otimes dt$$
and $(\tilde{\lambda}, \tilde{\phi}, \tilde{m})$ satisfies the optimality conditions

\begin{equation}
\left \{
\begin{array}{ll}
\displaystyle -\partial_t \tilde{\phi}(t,x) + H(t,x, D\tilde{\phi}(t,x),D^2\tilde{\phi}(t,x)) = f'_2(t,x) & \mbox{in } [0,T]\times \R^d\\
\displaystyle \partial_t \tilde{m} - \mdiv (\partial_p H(t,x,D \tilde{\phi}(t,x), D^2 \tilde{\phi}(t,x))\tilde{m}) \\
\displaystyle \hspace{60pt}+ \sum_{i,j} \partial_{ij}^2 ((\partial_M H(t,x,D \tilde {\phi} (t,x),D^2\tilde{\phi}(t,x)))_{ij}\tilde{m}) = 0 & \mbox{in } [0,T]\times \R^d \\
\displaystyle \tilde{\phi}(T,x) = \tilde{\lambda} h(x) + g'(x)  \mbox{   in   } \R^d \mbox{,   }
\displaystyle \tilde{m}(0) = m_0 \\
\displaystyle \tilde{\lambda} \int_{\R^d} h(x) d\tilde{m}(T)(x) = 0 \mbox{,       }
\displaystyle  \int_{\R^d} h(x) d\tilde{m}(T)(x) \leq 0 \mbox{,        }
\displaystyle \tilde{\lambda} \geq 0. \\
\end{array}
\right.
\label{OC2}
\end{equation}

\end{corollary}

\begin{proof}[Proof of Corollary \ref{OptimalityConditionsLinearConstraint}]

Let $(\phi,\lambda) \in \mathbb{A}$ and $(m,\omega,W) \in \mathbb{K}$ points where the primal and the dual problems are achieved. One has $\displaystyle \int_{\R^d} \phi(0,x)dm_0(x) = J_{RP}(m,\omega,W)$. Given the constraint on $\phi$ and the fact that $m$ is non-negative we get

\begin{align*}
\int_{\R^d} \phi(0,x) dm_0(x) &- \int_0^T \int_{\R^d} (-\partial_t \phi (t,x)+ H(t,x,D\phi(t,x),D^2\phi(t,x)))dm(t)(x)dt  \\
& \geq \int_0^T \int_{\R^d} \left[ L \left(t,x,\frac{d\omega}{dt \otimes dm}(t,x),\frac{dW}{dt\otimes dm}(t,x) \right) \right] dm(t)(x)dt + \int_{\R^d} g'(x)dm(T)(x). 
\end{align*}
Yet, $(m,\omega,W)$ solves the Fokker-Planck equation and $\phi(T,x) \leq \lambda h(x) + g'(x)$ for all $x \in \R^d$ so

\begin{align*}
& \lambda \int_{\R^d} h(x)dm(T)(x) - \int_0^T \int_{\R^d} D\phi (t,x). d\omega(t,x) - \int_0^T \int_{\R^d} D^2\phi(t,x).dW(t,x) \\
& \geq \int_0^T \int_{\R^d} \left [ L \left(t,x,\frac{d\omega}{dt\otimes dm}(t,x),\frac{dW}{dt \otimes dm}(t,x) \right) +  H(t,x,D\phi(t,x),D^2\phi(t,x)) \right ] dm(t)(x)dt. 
\end{align*}

\noindent Remember that $\displaystyle \int_{\R^d} h(x)dm(T)(x) \leq 0$ and $\lambda \geq 0$ so

\begin{align*}
& \int_0^T \int_{\R^d} [ L \left(t,x,\frac{d\omega}{dt \otimes dm}(t,x),\frac{dW}{dt\otimes dm}(t,x) \right) +  H(t,x,D\phi(t,x),D^2\phi(t,x)) \\
&- D\phi(t,x) .\frac{d\omega}{dt \otimes dm}(t,x)-D^2\phi(t,x).\frac{dW}{dt \otimes dm}(t,x)) ] dm(t)(x)dt \\
& \leq \lambda \int_{\R^d} h(x)dm(T)(x) \leq 0.
\end{align*}

\noindent But, by definition of $L$, the integrand is always nonnegative. So, $m(t)\otimes dt$-ae we have

\begin{align*} 
& - D\phi (t,x).\frac{d\omega}{dt\otimes dm}(t,x)-D^2\phi(t,x).\frac{dW}{dt \otimes dm}(t,x) \\
&= L \left(t,x,\frac{d\omega}{dt \otimes dm}(t,x),\frac{dW}{dt \otimes dm}(t,x) \right) + H(t,x, D\phi(t,x),D^2\phi(t,x))
\end{align*}
and since $H$ is differentiable, $m(t)\otimes dt$-ae it holds

\begin{equation*}
\left \{
\begin{array}{ll}
\frac{d\omega}{dt \otimes dm} (t,x)= -\partial_pH(t,x,D\phi(t,x),D^2\phi(t,x)) \\
 \frac{dW}{dt \otimes dm}(t,x) = -\partial_MH(t,x,D\phi(t,x),D^2\phi(t,x)). 
\end{array}
\right.
\end{equation*}
Finally, since all the inequalities at the beginning of this proof are actually equalities, we get the necessary conditions for optimality.
\end{proof}

\section{Proof of the main results}

\label{sec: Proof of the Main Theorem}

\subsection{Linearization}

Let us fix $(\tilde{m},\tilde{\omega},\tilde{W})$ a solution of the relaxed problem. The linearized problem is to minimize 

\begin{align}
\notag J^l_{RP}(m,\omega,W) &:= \int_0^T \int_{\R^d} L \left( t,x,\frac{d\omega}{dt \otimes dm}(t,x), \frac{dW}{dt \otimes dm}(t,x) \right) dm(t)(x)dt \\
& +\int_0^T\int_{\R^d} \frac{ \delta f_2}{\delta m}(t,\tilde{m}(t),x)dm(t)(x)dt +\int_{\R^d} \frac{\delta g}{\delta m}(\tilde{m}(T),x)dm(T)(x)
\label{Linearized Problem 2020}
\end{align}
among triples $(m,\omega,W)$ that satisfy the Fokker-Planck equation with $m(0)=m_0$ and with $m$ satisfying the linearized constraint 

\begin{equation}
\int_{\R^d} \frac{\delta \Psi}{\delta m}(\tilde{m}(T),x)dm(T)(x) \leq 0.
\label{linearized constraint}
\end{equation}
Notice that we are in the setting of Section \ref{sec: Necessary Conditions} with $\displaystyle f_2'(t,x)= \frac{\delta f_2}{\delta m}(t,\tilde{m}(t),x)$, $\displaystyle g'(x)= \frac{\delta g}{\delta m}(\tilde{m}(T),x)$ and $\displaystyle h(x) = \frac{\delta \Psi}{\delta m}(\tilde{m}(T),x)$.

\begin{proposition}
Let $(\tilde{m}, \tilde{\omega}, \tilde{W})$ be a fixed solution to the relaxed problem. If $\Psi(\tilde{m}(T))=0$ then $(\tilde{m},\tilde{\omega},\tilde{W})$ is a solution of the linearized problem \ref{Linearized Problem 2020}. If $\Psi(\tilde{m}(T)) <0$ then $(\tilde{m},\tilde{\omega},\tilde{W})$ is a solution of the linearized problem \ref{Linearized Problem 2020} without the final constraint.
\label{From general to linear constraint}
\end{proposition}

\begin{proof}
Suppose that $\Psi(\tilde{m}(T)) =0$. By condition \ref{TransCondU} there is some $x_0 \in \R^d$ such that $\displaystyle \frac{\delta \Psi}{\delta m}(\tilde{m}(T),x_0) <0$ and we can proceed as in Lemma \ref{Reachability} (the constraint being then the linear one: $\displaystyle \tilde{\Psi}(m) = \int_{\R^d} \frac{\delta \Psi}{\delta m}(\tilde{m}(T),x)dm(x)$) and find $(m',\omega',W')$ such that 
$$\left \{ 
\begin{array}{ll}
\displaystyle m'(0) = m_0 \\
\displaystyle \partial_t m' + \mdiv(\omega') - \sum_{i,j} \partial_{ij}^2 W'_{ij} =0 \\
\displaystyle \int_{\R^d} \frac{\delta \Psi}{\delta m}(\tilde{m}(T),x)dm'(T)(x) < 0 \\
\displaystyle J^l_{RP}(m', \omega', W') < +\infty. 
\end{array}
\right.
$$
Now let $(m,\omega,W)$ be any candidate for the linearized problem (in particular $(m,\omega,W)$ satisfies the linearized constraint \ref{linearized constraint}). Let $\epsilon \in (0,1)$ and define $(m^{\epsilon}, \omega^{\epsilon}, W^{\epsilon}):= (1-\epsilon) (m,\omega,W) + \epsilon (m',\omega',W')$ (we perturb $(m, \omega,W)$ a little bit so that it satisfies strictly the linearized constraint). Let  $\lambda \in (0,1)$ and define $(m_{\lambda}^{\epsilon}, \omega_{\lambda}^{\epsilon}, W_{\lambda}^{\epsilon}):= (1-\lambda) (\tilde{m},\tilde{\omega},\tilde{W}) + \lambda (m^{\epsilon}, \omega^{\epsilon}, W^{\epsilon}) $. We have that 

\begin{equation}
\Psi( m_{\lambda}^{\epsilon}(T) ) = \Psi( \tilde{m}(T) ) + \lambda \int_{\R^d} \frac{\delta \Psi}{\delta m}(\tilde{m}(T),x) dm^{\epsilon}(T)(x) + \circ (\lambda)
\label{DLenlambda}
\end{equation}
but

$$\int_{\R^d}\frac{\delta \Psi}{\delta m}(\tilde{m}(T),x) dm^{\epsilon}(T)(x) = (1-\epsilon) \int_{\R^d}\frac{\delta \Psi}{\delta m}(\tilde{m}(T),x) dm(T)(x) +\epsilon \int_{\R^d}\frac{\delta \Psi}{\delta m}(\tilde{m}(T),x) dm'(T)(x) <0 $$
and therefore $\Psi( m_{\lambda}^{\epsilon}(T) ) \leq 0$ for small enough $\lambda$. Now, by convexity of
 
$$ (m,\omega,W) \rightarrow \Gamma (m,\omega,W) := \int_0^T \int_{\R^d} L \left( t,x,\frac{d\omega}{dt \otimes dm}(t,x), \frac{dW}{dt \otimes dm}(t,x) \right) dm(t)(x)dt $$ 
and optimality of $(\tilde{m},\tilde{\omega},\tilde{W})$ for the linearized problem we have 

\begin{align*}
\Gamma(\tilde{m},\tilde{\omega},\tilde{W}) &\leq \Gamma (m_{\lambda}^{\epsilon}, \omega_{\lambda}^{\epsilon}, W_{\lambda}^{\epsilon}) +\int_0^T \left[ f_2(t, m_{\lambda}^{\epsilon}(t)) - f_2(t,\tilde{m}(t)) \right] dt +  g( m_{\lambda}^{\epsilon}(T)) - g(\tilde{m}(T))  \\
&\leq (1-\lambda) \Gamma(\tilde{m},\tilde{\omega},\tilde{W}) + \lambda \Gamma (m^{\epsilon}, \omega^{\epsilon}, W^{\epsilon}) +\int_0^T \left[ f_2(t, m_{\lambda}^{\epsilon}(t)) - f_2(t,\tilde{m}(t)) \right] dt \\
&+ g( m_{\lambda}^{\epsilon}(T)) - g(\tilde{m}(T))  
\end{align*}
which gives 
$$ \Gamma (\tilde{m},\tilde{\omega},\tilde{W}) \leq \Gamma (m^{\epsilon}, \omega^{\epsilon}, W^{\epsilon}) + \frac{1}{\lambda}\int_0^T \left[ f_2(t, m_{\lambda}^{\epsilon}(t)) - f_2(t,\tilde{m}(t)) \right] dt + \frac{1}{\lambda} \left[ g( m_{\lambda}^{\epsilon}(T)) - g(\tilde{m}(T)) \right]. $$
Now we let $\lambda$ go to $0$ and use once again the convexity of $\Gamma$ to get
\begin{align*}
J_{RP}^l( \tilde{m}, \tilde{\omega},\tilde{W}) &= \Gamma ( \tilde{m}, \tilde{\omega},\tilde{W})  \\
& \leq \Gamma (m^{\epsilon}, \omega^{\epsilon}, W^{\epsilon}) + \int_0^T \int_{\R^d} \frac{\delta f_2}{\delta m}(t,\tilde{m}(t),x)dm^{\epsilon}(t)(x)dt + \int_{\R^d} \frac{\delta g}{\delta m}(\tilde{m}(T),x)dm^{\epsilon}(T)(x) \\
& \leq J_{RP}^l (m,\omega,W) + \epsilon \left( J_{RP}^l (m',\omega',W') - J_{RP}^l(m,\omega,W) \right). 
\end{align*}
We get the result letting $\epsilon \rightarrow 0$. When $\Psi(\tilde{m}(T)) <0$ there is no need to perturb $(m,\omega,W)$ since \ref{DLenlambda} shows that $\Psi( m_{\lambda}^{0}(T) ) \leq 0$ for small enough $\lambda$ independently from the sign of $\displaystyle \int_{\R^d} \frac{\delta \psi}{\delta m}(\tilde{m}(T),x)dm(T)(x)$ and we can take $\epsilon =0$ in the rest of the proof.
\end{proof}

\subsection{General constraint}

\begin{proof}[Proof of Theorem \ref{MainGeneral}]

Recall that, on the one hand we want to prove the existence of optimal Markovian controls for \ref{SP} and on the other hand we want to prove that optimal controls, if Markovian, satisfy some necessary conditions. Let $(\tilde{m},\tilde{\omega},\tilde{W})$ be a solution of the relaxed problem. We can apply Proposition \ref{From general to linear constraint} and Corollary \ref{OptimalityConditionsLinearConstraint} to find some $(\tilde{\lambda}, \tilde{\phi})$ in $\R^+ \times \cC^{1,2}_b([0,T]\times \R^d)$ such that $(\tilde{m},\tilde{\lambda},\tilde{\phi})$ satisfies the system of optimality conditions \ref{OC2} with $\displaystyle f_2'(t,x)= \frac{\delta f_2}{\delta m}(t,\tilde{m}(t),x)$, $\displaystyle g'(x)= \frac{\delta g}{\delta m}(\tilde{m}(T),x)$ and $\displaystyle h(x) = \frac{\delta \Psi}{\delta m}(\tilde{m}(T),x)$. Notice that, when $\Psi(\tilde{m}(T)) < 0$ we can take $\lambda=0$ since $(\tilde{m},\tilde{\omega},\tilde{W})$ is a solution of the linearized problem without constraint in this case.  In general, let $\tilde{\alpha}$ be a measurable function such that, for all $(t,x) \in[0,T] \times \R^d$

\begin{align*}
H(t,x,D\tilde{\phi}(t,x),D^2\tilde{\phi}(t,x)) = &-b(t,x, \tilde{\alpha}(t,x)) . D\tilde{\phi}(t,x) - \sigma^t\sigma(t,x, \tilde{\alpha}(t,x)) . D^2\tilde{\phi}(t,x) \\
&-f_1(t,x, \tilde{\alpha}(t,x)). 
\end{align*}

\noindent We use the assumption that $H$ is continuously differentiable in $(p,M)$. Indeed, in this case one has, thanks to the  Envelope theorem (see \cite{Milgrom2002}),  

\begin{equation}
\label{Envelope}
\left \{
\begin{array}{ll}
 \partial_p H(t,x,D\tilde{\phi}(t,x),D^2\tilde{\phi}(t,x)) = -b(t,x, \tilde{\alpha}(t,x))  \\
\partial_M H(t,x,D\tilde{\phi}(t,x),D^2\tilde{\phi}(t,x))= -\sigma ^t\sigma(t,x, \tilde{\alpha}(t,x)).  
\end{array}
\right .
\end{equation}

\noindent Since $\partial_p H$ and $\partial_M H$ are supposed to be locally Lipschitz continuous respectively in $p$ and $M$ and uniformly in $x$ and since $|\partial_M H|$ is bounded from below by $\sqrt{d}\Lambda^- >0$, using the fact (Theorem \ref{MainHJB}) that $\tilde{\phi}$ belongs to $\cC_b^{\frac{3+\alpha}{2}, 3+ \alpha}([0,T] \times \R^d)$ we see that the coefficients of the functions, $(t,x) \rightarrow \partial_pH(t,x,D\tilde{\phi}(t,x),D^2\tilde{\phi}(t,x))$ and  $(t,x) \rightarrow \partial_M H(t,x,D\tilde{\phi}(t,x),D^2\tilde{\phi}(t,x))$ are Lipschitz in $x$, uniformly in $t$. Thus there is a unique strong solution of the SDE

$$d\tilde{X}_t = b(t,X_t,\alpha(t,X_t))dt  + \sqrt{2} \sigma(t,X_t,\alpha(t,X_t))dB_t$$
starting from $X_0$. Therefore $\cL(\tilde{X}_t) = \tilde{m}(t) $ for all $t \in [0,T]$ and, in particular, $\Psi(\cL(\tilde{X}_T)) \leq 0 $. This means that $\tilde{\alpha}_t := \tilde{\alpha}(t,\tilde{X}_t)$ is admissible for the strong problem. Since $H$ is $\mathcal{C}^1$ we know that for all $(t,x,p,M) \in [0,T] \times \R^d \times \R^d \times \mathbb{S}_d(\R)$, 

\begin{align}
\notag H(t,x,p,M) &= p.\partial_p H(t,x,p,M) + M.\partial_M H(t,x,p,M) \\
&- L(t,x,-\partial_p H(t,x,p,M) , -\partial_M H(t,x,p,M))
\label{Pasdideedenom}
\end{align}
and therefore, \eqref{Envelope} implies that

$$L(t,x,-\partial_p H(t,x,D\tilde{\phi}(t,x),D^2\phi(t,x)),-\partial_M H(t,x,D\phi(t,x),D^2\phi(t,x))) = f_1(t,x,\alpha(t,x))$$
and thus $J_{SP}(\tilde{\alpha}) = J_{RP}(\tilde{m}, \tilde{\omega}, \tilde{W}) = V_{SP}(X_0)$ from which it comes that $V_{RP} (m_0) \geq V_{SP} (X_0)$. The reverse inequality being clear, we get $V_{RP}(m_0) = V_{SP}(X_0) $ and $\tilde{\alpha}$ is a solution to the strong problem. This shows in particular that optimal controls for the strong problem \ref{SP} do exist. Now take $\alpha$ a Markovian solution to the strong problem. If $X$ is the corresponding process, we take $(m, \omega, W) = (m, b(x,\alpha^1) m, \sigma^t\sigma(x,\alpha^2)m) $. Then, $(m, \omega, W)$ is admissible for the relaxed problem and we have $J_{RP} (m, \omega, W) \leq J_{SP} (\alpha^1, \alpha^2) = V_{SP}(X_0)$. And thus, $J_{RP} (m, \omega, W) = V_{RP}(m_0) $.  Finally, $(m, \omega, W)$ is optimal for the relaxed problem and we can apply Proposition \ref{From general to linear constraint} and  Corollary \ref{OptimalityConditionsLinearConstraint} to conclude. Now if we use $\tilde{\phi}$ in \ref{OC} as a test function for the Fokker-Planck equation, recalling \ref{Pasdideedenom} as well as the convention $\displaystyle \int_{\R^d} \frac{\delta U}{\delta m}(\mu,x)d\mu(x) =0$ for the linear functionnal derivative we get that 

\begin{align*}
& \int_{\R^d} \tilde{\phi}(0)dm_0(x) \\
&= \int_0^T \int_{\R^d} L \left( t,x,-\partial_pH(t,x,D\tilde{\phi}(t,x),D^2\tilde{\phi}(t,x)), -\partial_MH(t,x,D\tilde{\phi}(t,x),D^2\tilde{\phi}(t,x)) \right) d\tilde{m}(t)(x)dt
\end{align*}
and therefore $\displaystyle V_{SP}(X_0) = \int_{\R^d} \tilde{\phi}(0)dm_0(x) + \int_{\R^d} f_2(t,\tilde{m}(t))dt + g(\tilde{m}(T))$.

\end{proof}

\subsection{Convex constraint and convex costs}

Now we show that the conditions are also sufficient when $\Psi$, $f_2$ and $g$ are convex functions in the measure variable. Notice that this case covers in particular the problem with expectation constraint and costs in expectation form when $\Psi$, $f_2$ and $g$ are linear.

\begin{proof}[Proof of Theorem \ref{MainConv}]

\noindent Let $(\tilde{\lambda}, \tilde{\phi}, \tilde{m})$ be a solution to the system of optimality conditions \ref{OC} and let $\tilde{X}_t$ be the solution to

\begin{equation*}
\left \{
\begin{array}{ll}
 d\tilde{X}_t = b(t,X_t,\tilde{\alpha}(t,\tilde{X}_t)) dt + \sqrt{2}\sigma(t,\tilde{X}_t,\tilde{\alpha}(t,X_t))dB_t \\
\tilde{X}_0 = X_0. 
\end{array}
\right.
\end{equation*}
for some measurable function $\tilde{\alpha} : [0,T] \times \R^d \rightarrow A$ such that, for any $(t,x) \in [0,T] \times \R^d $, 

$$H(t,x,D\tilde{\phi}(t,x),D^2\tilde{\phi}(t,x)) =-b(t,x,\tilde{\alpha}(t,x)).D\tilde{\phi}(t,x) - \sigma ^t\sigma(t,x,\tilde{\alpha}(t,x)).D^2\tilde{\phi}(t,x) - f_1(t,x,\tilde{\alpha}(t,x)). $$
Since $b(t,x,\tilde{\alpha}(t,x)) = -\partial_pH(t,x,D\tilde{\phi}(t,x),D^2\tilde{\phi}(t,x))$, $\sigma ^t\sigma(t,x,\tilde{\alpha}(t,x)) = -\partial_M H(t,x,D\tilde{\phi}(t,x),D^2\tilde{\phi}(t,x))$ and $\phi$ belongs to $\mathcal{C}_b^{3+\alpha, \frac{3+\alpha}{2}}([0,T] \times \R^d)$ the SDE admits a unique strong solution.

\noindent We are going to show that $\tilde{\alpha}_t := \tilde{\alpha}(t,\tilde{X}_t) $ is a solution to the optimal control problem. The law of $\tilde{X}_t$ is $\tilde{m}(t)$ and we deduce that $\Psi(\cL(\tilde{X}_T)) \leq 0 $ and $\tilde{\alpha}_t$ is admissible. Now we show that $\tilde{\alpha}_t$ is indeed optimal among the admissible strategies. Let $\alpha_t$ be an admissible control, $X_t$ the corresponding process and $m(t):= \mathcal{L}(X_t)$. Let also $J_{SP}'$ be defined on $\mathcal{U}_{ad}$ as follows

$$ J_{SP}'(\alpha_t) := \E \left( \int_0^T \left(f_1(t,X_t,\alpha_t) + \frac{\delta f_2}{\delta m}(t,\tilde{m}(t),X_t)\right)dt + \frac{\delta g}{\delta m}(\tilde{m}(t),X_T) + \tilde{\lambda} \frac{\delta \Psi}{\delta m}(\tilde{m}(T),X_T) \right). $$

\noindent Using a classical verification argument and the fact that $\tilde{\phi}$ solves the HJB equation, we get that $ J_{SP}' (\tilde{\alpha}_t) \leq J_{SP}'(\alpha_t)$. Now by convexity of $\Psi$, $f_2$ and $g$ we get

$$\mathbb{E} \left[ \int_0^T f_2(t,\tilde{m}(t))dt - \int_0^T f_2(t,m(t))dt + \int_0^T \frac{\delta f_2}{\delta m}(t,\tilde{m}(t),X_t)dt \right] \leq 0,$$

$$\mathbb{E} \left[ g(\tilde{m}(T)) - g(m(T)) + \frac{\delta g}{\delta m}(\tilde{m}(T),X_T) \right] \leq 0$$
and

$$\tilde{\lambda} \mathbb{E} \left[ \frac{\delta \Psi}{\delta m}(\tilde{m}(T),X_T) \right] = \tilde{\lambda} \left( \Psi(\tilde{m}(T)) + \mathbb{E} \left[ \frac{\delta \Psi}{\delta m}(\tilde{m}(T),X_T) \right] \right) \leq \tilde{\lambda}\Psi(m(T)) \leq 0.$$
Therefore we get that $J_{SP}(\tilde{\alpha}) \leq J_{SP}(\alpha)$ and $\tilde{\alpha}$ is optimal for the strong problem.
\end{proof}

\section{The HJB equation}

\label{sec: HJB}

The aim of this section is to show that the HJB equation 
\begin{equation}
\left \{
\begin{array}{ll}
\displaystyle -\partial_t u(t,x) + H(t,x, Du(t,x),D^2u(t,x)) = f_2'(t,x) & \mbox{in } [0,T]\times \R^d\\
u(T,x) = g'(x) & \mbox{in } \R^d
\end{array}
\right.
\label{HJBFinalSection}
\end{equation}

\noindent admits a unique strong solution $u \in \cC_b^{ \frac{3+\alpha}{2}, 3+\alpha}([0,T] \times \R^d)$. We first recall that, by classical arguments (see \cite{Fleming2006} Chapter $V$), the equation \eqref{HJBFinalSection} satisfies a comparison principle between bounded uniformly continuous sub and super-solutions and admits therefore a unique bounded uniformly continuous viscosity solution. We denote by $u$ this solution and now observe that it is enough to prove that $u$ is Lipschitz continuous in $(t,x)$ to deduce that it is actually in $\cC_b^{ \frac{3+\alpha}{2}, 3+\alpha}([0,T] \times \R^d)$. Indeed, if $u$ is Lipschitz continuous in space we can use Theorem VII.3 in \cite{Ishii1990} to deduce that $u$ is semi-concave with a modulus of semi-concavity uniform in $(t,x)$. Now using the strict parabolicity of the equation, the fact that $u$ is Lipschitz and semi-concave we can prove that $u$ is also semi-convex (see \cite{Imbert2006} Theorem 4 with the help of \cite{Alvarez1997} ) and therefore $Du$ is continuous and Lipschitz in space. At this point, using the Hölder regularity in time of $f_2'$, we can use the results of  \cite{Wang1992a} and \cite{Wang1992b} (see also the last section of \cite{Bourgoing2004}) to deduce that $u$ belongs to $\mathcal{C}_b^{1,2}([0,T] \times \R^d)$. Finally, by differentiating the equation we can use results on uniformly parabolic linear PDEs (Theorem IV.5.1 of \cite{Ladyzenskaja1968}) to conclude that $u$ belongs to $\mathcal{C}^{\frac{3+\alpha}{2},3+\alpha}_b([0,T] \times \R^d)$. 

Now we proceed to show that $u$ is indeed globally Lipschitz continuous. We first show this when $f_2'$ is also globally Lipschitz continuous (and not just Hölder continuous in time) and then we use an approximation argument.

\begin{lemma}
Suppose that Assumptions \ref{IA} and \ref{HAss} hold. Take $g' \in \mathcal{C}_b^{3+\alpha}(\R^d)$ and suppose that $f_2' \in \mathcal{C}_b([0,T] \times \R^d)$ is globally Lipschitz continuous in $(t,x)$ and $\mathcal{C}^1$ in $x$.  Let $u$ be the unique viscosity solution to \eqref{HJBFinalSection}. Then $u$ is also globally Lipschitz-continuous.
\end{lemma}

\begin{proof}
We first show the regularity in time. We observe that, since $g' \in \cC_b^{2}(\R^d)$,  for 
$$C_{g'} \geq \sup_{(t,x)} |H(t,x,Dg'(t,x),D^2g'(t,x)) -f_2'(t,x)|, $$
$g'(x) - C_{g'} (T-t)$ and $g'(x) + C_{g'}(T-t)$ are respectively viscosity sub-solution and super-solution to \eqref{HJBFinalSection}. By comparison we have that $|u(T-t,x) -g'(x)|\leq C_{g'} t$ for all $t \in [0,T]$. If we fix $s \in [0,T]$ and define for all $(t,x) \in [s,T] \times \R^d$, $v(t,x) = u(t-s,x)$ it is plain to check that $v^+(t,x) := v(t,x) - C'st$ and $v^-(t,x):=v(t,x)+C'st$ are respectively sub and super solutions where $C'$ is such that 

$$ |H(t-s,x,p,M) - H(t,x,p,M)| + |f_2'(t-s,x) - f_2'(t,x) | \leq C's $$
for all $s \in [0,T]$, all $t \in [s,T]$ and all $(x,p,M) \in \R^d \times \R^d \times \mathbb{S}_d(\R)$. By comparison we find that for all $s \in [0,T]$, all $t \in [s,T]$ and all $x \in \R^d$,
\begin{align*}
u(t,x) - v^+(t,x) &\leq \sup_{x \in \R^d} u(T,x) - v^+(T,x) \\
&\leq \sup_{x \in \R^d} g'(x) - u(T-s,x) + C'Ts \leq C_{g'} s + C' T s.
\end{align*}
Doing the same with $v^-$ we get that $|u(t,x) - u(t-s,x)| \leq (C_{g'} + C'T)s$ for all $s \in [0,T]$, all $t \in [s,T]$ and all $x \in \R^d$.

Now we show the space regularity. Let $K > \|Dg'\|_{\infty}$ such that $H(t,x,p,0) > L_T + \|f_2' \|_{\infty}$ for all $(t,x,p) \in [0,T] \times \R^d \times \R^d$ such that $|p| \geq K-1$ and where $L_T$ is an upper bound for the time-Lipschitz constant of $u$. We are going to show that for all $(t,x,y) \in [0,T]\times \R^d\times \R^d$, $ u(t,x) - u(t,y) \leq K|x-y|$ when $K$ is large enough.  Suppose on the contrary that $ \delta := \sup_{(t,x,y) \in \{[0,T]\times \R^d\times \R^d} \{u(t,x) - u(t,y) - K|x-y| \}$ is positive. Let $\beta$ be a small positive parameter and define

$$\phi_{\beta}(t,x,y) := u(t,x) - u(t,y) - K|x-y| - \beta|y|^2-\beta\frac{1}{t}.$$

\noindent The function $\phi_{\beta}$ reaches its maximum at some point $(\bar{t},\bar{x},\bar{y}) \in (0,T]\times \R^d\times \R^d $ and there is $\beta_0 >0$ such that for $0 < \beta \leq \beta_0$

\begin{equation}
\phi_{\beta}(\bar{t},\bar{x},\bar{y}) \geq \frac{\delta}{2}.
\label{In1}
\end{equation}

\noindent Suppose that $\beta \leq \beta_0$ and $\bar{t}=T$, then

\begin{align*}
\frac{\delta}{2} \leq \phi_{\beta}(T,\bar{x},\bar{y}) &= u(T,\bar{x}) - u(T, \bar{y}) -K|\bar{x}-\bar{y}| - \beta|\bar{y}|^2 -\frac{\beta}{T} \\
& \leq (\|Dg'\|_{\infty} - K)|\bar{x}-\bar{y}|.
\end{align*}

\noindent But this is impossible since $K > \|Dg'\|_{\infty}$ and $\delta >0$. Thus for all  $\beta \leq \beta_0$, $\bar{t} \neq T$. From \eqref{In1} we deduce that $\beta |\bar{y}|^2 \leq 2\|u\|_{\infty} $ and thus that $\beta|\bar{y}| \rightarrow 0$ as $\beta \rightarrow 0$ but we also deduce that  

$$ \frac{\delta}{2} \leq u(\bar{t},\bar{x}) - u(\bar{t},\bar{y}) - K|\bar{x}-\bar{y}|, $$
in particular, $\bar{x} \neq \bar{y}$.  Since $\bar{t} \neq T$ for $\beta \leq \beta_0$, we can apply the maximum principle for semi-continuous functions from \cite{Crandall1992}. Let $\varphi_{\beta} (t,x,y)= K|x-y| +\beta \frac{1}{t}$. Computing the various derivatives for $|x-y| >0$ gives 

\begin{equation*}
\left \{
\begin{array}{ll}
\partial_t \varphi_{\beta}(t,x,y) = - \frac{\beta}{t^2} \\
\partial_x \varphi_{\beta}(t,x,y) = \frac{K}{|x-y|} (x-y) \\
\partial_y \varphi_{\beta}(t,x,y) = \frac{K}{|x-y|} (y-x) \\
D^2\varphi_{\beta}(t,x,y) = \frac{K}{|x-y|} \begin{pmatrix}
  I_d & -I_d\\ 
  -I_d & I_d
\end{pmatrix} 
- \frac{K}{|x-y|^3} \begin{pmatrix} (x-y)\otimes(x-y) & -(x-y)\otimes(x-y) \\ -(x-y)\otimes(x-y) & (x-y)\otimes(x-y) \end{pmatrix}.
\end{array}
\right.
\end{equation*}

\noindent In particular, if $N:= (x-y)\otimes(x-y)$, then $N \geq 0$ (rank one symetric matrice with positive trace) and thus it is elementary to show that $\begin{pmatrix} N & -N \\ -N & N \end{pmatrix} \geq 0$. And thus, 
 
 $$ D^2\varphi_{\beta}(t,x,y) \leq \frac{K}{|x-y|} \begin{pmatrix} I_d & -I_d\\ -I_d & I_d \end{pmatrix}.  $$
 
 \noindent Now, from the maximum principle, we get $\nu \in \R$, $X,Y \in \mathbb{S}_d(\R) $ such that

\begin{equation*}
\left \{
\begin{array}{ll}
\left(\nu, K \frac{\bar{x}-\bar{y}}{|\bar{x}-\bar{y}|},X \right) \in \bar{\mathcal{P}}^{2,+}u(\bar{t},\bar{x}) \\
\left(\nu + \frac{\beta}{\bar{t}^2}, K \frac{\bar{x}-\bar{y}}{|\bar{x}-\bar{y}|}-2\beta\bar{y},Y \right) \in \bar{\mathcal{P}}^{2,-}u(\bar{t},\bar{y}) \\
\end{array}
\right.
\end{equation*}

\noindent and

\begin{equation*}
\begin{pmatrix} X & 0 \\ 0 & -(Y+ 2 \beta I_d) \end{pmatrix} \leq 3 \frac{K}{|\bar{x}-\bar{y}|} \begin{pmatrix} I_d & -I_d \\ -I_d & I_d \end{pmatrix}.
\end{equation*}

\noindent Observe that $|\nu|$ is bounded by $L_T$ the time -Lipschitz constant of $u$ and thus $|\nu|$ is bounded independently of $K$. Now we use the equation satisfied by $u$ to get

\begin{equation*}
H \left(\bar{t},\bar{x}, K\frac{\bar{x}-\bar{y}}{|\bar{x}-\bar{y}|},X \right) - f_2'(\bar{t}, \bar{x}) \leq \nu \leq H \left(\bar{t},\bar{y}, K\frac{\bar{x}-\bar{y}}{|\bar{x}-\bar{y}|}- 2\beta \bar{y},Y \right) - f_2'(\bar{t},\bar{y}).
\end{equation*}

\noindent From now on, we let $ \bar{\xi} := K\frac{\bar{x}-\bar{y}}{|\bar{x}-\bar{y}|}$ and $ \gamma := \frac{3K}{|\bar{x}-\bar{y}|}$. We are going to show that the information

\begin{equation}
\left \{
\begin{array}{ll}
|\nu| \leq L_T \\
|\bar{\xi} |= K \\
\begin{pmatrix} X & 0 \\ 0 & -(Y+ 2 \beta I_d) \end{pmatrix} \leq \gamma \begin{pmatrix} I_d & -I_d \\ -I_d & I_d \end{pmatrix} \\
H \left(\bar{t},\bar{x}, \bar{\xi},X \right) - f_2'(\bar{t}, \bar{x})\leq \nu \leq H \left(\bar{t},\bar{y}, \bar{\xi}- 2\beta \bar{y},Y \right)- f_2'(\bar{t},\bar{y}) \\
\end{array}
\right.
\label{Data1}
\end{equation}

\noindent is inconsistent whenever $K$ is sufficiently large. Let $\bar{\eta} := \bar{\xi} - 2\beta \bar{y}$ and for any $\lambda \in [0,1] $, $x_{\lambda} := (1-\lambda)\bar{x} + \lambda \bar{y}$ and $\xi_{\lambda} := (1-\lambda)\bar{\xi} + \lambda \bar{\eta} = \bar{\xi}- 2\lambda \beta \bar{y}$. From \cite{Armstrong2018}, Lemma A.2, there exists a $\mathcal{C}^1$ map, $\lambda \rightarrow Z_{\lambda}$ from $[0,1] \rightarrow \mathbb{S}_d(\R) $ such that

\begin{equation*}
\left \{
\begin{array}{ll}
 \frac{d}{d\lambda}Z_{\lambda} = \gamma^{-1} Z_{\lambda}^2, \\
Z_0 = X \\
\forall \lambda \in [0,1]$, $X \leq Z_{\lambda} \leq Y + 2\beta I_d. 
\end{array}
\right.
\end{equation*}

Let us define $l : [0,1] \rightarrow \R$ by $ l(\lambda) = H(\bar{t},x_{\lambda}, \xi_{\lambda}, Z_{\lambda}) -f_2'(\bar{t},x_{\lambda})$, so that $ l(0) = H \left(\bar{t},\bar{x}, \bar{\xi},X \right) -f_2'(\bar{t},\bar{x})\leq \nu $ and (using $Z_1 \leq Y + 2\beta I_d$ and the boundness of $\partial_M H$)

$$l(1) = H(\bar{t},\bar{y}, \bar{\eta}, Z_1) \geq H(\bar{t},\bar{y}, \bar{\eta}, Y+2\beta I_d) -f_2'(\bar{t}, \bar{y})\geq H(\bar{t},\bar{y}, \bar{\eta},Y) - C\beta -f_2'(\bar{t}, \bar{y})\geq \nu - C\beta,$$
where $C = 2\Lambda^+ \sqrt{d} $. Thus, $l$ being $\mathcal{C}^1$, there exists $\lambda \in [0,1]$ such that 

\begin{align}
l(\lambda) = \nu- C\lambda\beta,
\label{Data2} \\
l'(\lambda) \geq -C\beta.
\label{Data3}
\end{align}

\noindent From inequality \eqref{Data2} we are going to obtain a lower bound on $|Z_{\lambda} | = \sqrt{\Tr(Z_{\lambda}^2)}$ and from inequality \eqref{Data3} we are going to obtain an upper bound on  $|Z_{\lambda} |$. Combining the two bounds will get a contradiction for $K$ large enough. First we exploit \eqref{Data2}. It gives us

\begin{align*}
H(\bar{t},x_{\lambda},\xi_{\lambda}, 0) -f_2'(\bar{t}, x_{\lambda}) - \nu +C\lambda \beta &= H(\bar{t},x_{\lambda},\xi_{\lambda}, 0) - H(\bar{t}, x_{\lambda}, \xi_{\lambda},Z_{\lambda}) \\
&\leq -\partial_M H(\bar{t}, x_{\lambda}, \xi_{\lambda},Z_{\lambda}) . Z_{\lambda} \\
&\leq \sqrt{d} \Lambda^+ |Z_{\lambda}|
\end{align*}
where we used Cauchy-Scharwz inequality at the last step. Therefore we have

\begin{equation}
|Z_{\lambda} | \geq \frac{H(\bar{t},x_{\lambda},\xi_{\lambda},0) -f_2'(\bar{t}, x_{\lambda})- \nu +C\lambda \beta}{\sqrt{d}\Lambda^+}. 
\label{Data2'}
\end{equation}

\noindent Now we use \eqref{Data3}. Computing the derivative of $l$ gives

\begin{align*} 
l'(\lambda) &= \partial_xH(\bar{t},x_{\lambda},\xi_{\lambda},Z_{\lambda}).(\bar{y}-\bar{x}) -\partial_xf_2'(\bar{t},x_{\lambda}).(\bar{y}-\bar{x})- 2\beta \partial_p H(\bar{t},x_{\lambda},\xi_{\lambda},Z_{\lambda}).\bar{y} \\
&+ \gamma^{-1} \partial_MH(\bar{t},x_{\lambda},\xi_{\lambda},Z_{\lambda}).Z_{\lambda}^2 \geq -C\beta 
\end{align*}
and since $\displaystyle -\partial_MH \geq \Lambda^-I_d$ we get

\begin{equation}
|Z_{\lambda}|^2 \leq \frac{1}{\Lambda^-} \left [ \gamma C\beta + \gamma \partial_xH(\bar{t},x_{\lambda},\xi_{\lambda},Z_{\lambda}).(\bar{y}-\bar{x}) -\gamma \partial_xf_2'(\bar{t},x_{\lambda}).(\bar{y}-\bar{x})- 2\beta \gamma \partial_p H(\bar{t},x_{\lambda},\xi_{\lambda},Z_{\lambda}).\bar{y} \right].
\label{Data3'}
\end{equation}
Recalling that $\gamma (\bar{x}-\bar{y}) = 3 \bar{\xi} = 3 \xi_{\lambda} +6 \lambda \beta \bar{y}$, combining \ref{Data2'} and \ref{Data3'} and using Assumption \ref{DerGrDxMHA}, we get that for some new positive constant $C$ independent from $(K,\delta,\beta,\lambda)$,

$$H(\bar{t},x_{\lambda},\xi_{\lambda},0)^2 \leq C \left( 1 +\nu^2 + |\xi_{\lambda}|^2 +\xi_{\lambda} .\partial_xH(\bar{t}, x_{\lambda}, \xi_{\lambda},0)+ \gamma \beta \bar{y}.\partial_p H(\bar{t},x_{\lambda},\xi_{\lambda},Z_{\lambda}) \right). $$
We get a contradiction letting $\beta \rightarrow 0$ as soon as $K$ is big enough since $|\xi| =K$, $H(t,x,p,0) \geq \alpha_1 |p|^{r_1} - C_H$ with $r_1>1$ for all $(t,x,p)$ and $\partial_xH$ satisfies Assumption \ref{DerGrDxHA}.
\end{proof}

To conclude with the proof we need to show Lipschitz estimates which are independent from the time regularity of $f_2'$. This is what we do in the following proof of Theorem \ref{MainHJB}.

\begin{proof}[Proof of Theorem \ref{MainHJB}]
When $f_2'=0$, the previous lemma and the discussion at the beginning of this section are enough to conclude. In the general case, take a smooth kernel $\rho$ with support in $[-1,1]$ and define for all $n \in \mathbb{N}^*$, $\rho_n(r) := n\rho(nr)$ and $\displaystyle f'^{(n)}_2(t,x) := \int_{-1}^1 f_2'(s,x) \rho_n(t-s)ds$ where we extended $f_2'$ to $[-1,T+1] \times \R^d$ by $f_2'(t,x) = f_2'(0,x)$ for $t \in [-1,0]$ and $f_2'(t,x) = f_2'(T,x)$ for $t\in [T, T+1]$. We also define $u_n$ to be the viscosity solution to 

\begin{equation*}
\left \{
\begin{array}{ll}
\displaystyle -\partial_t u_n(t,x) + H(t,x, Du_n(t,x),D^2u_n(t,x)) = f'^{(n)}_2(t,x) & \mbox{in } [0,T]\times \R^d\\
u(T,x) = g'(x) & \mbox{in } \R^d.
\end{array}
\right.
\end{equation*}
Thanks to the previous lemma and the discussion at the beginning of this section we know that $u_n$ actually belongs to $\mathcal{C}_b^{\frac{3+\alpha}{2},3+\alpha}([0,T] \times \R^d)$. If we define $w_n := \frac{1}{2}e^{\mu t} \left |Du_n \right |^2$ for some $\mu >0$ we get, after differentiating the HJB equation and taking scalar product with $e^{\mu t}Du_n$:

\begin{align*}
-\partial_tw_n &+ D_pH.Dw_n + D_MH.D^2w_n = Df_n.Du_ne^{\mu t} - D_xH.Du_ne^{\mu t} + e^{\mu t}D_MH.(D^2u_n)^2 -  \frac{1}{2}\mu e^{\mu t}|Du_n|^2 \\
& \leq Df'^{(n)}_2.Du_n e^{\mu t} + C_{D_xH}(1 + |Du_n| + |D^2u_n|)e^{\mu t} |Du_n| - e^{\mu t} \Lambda^- |D^2u_n|^2 -  \frac{1}{2}\mu e^{\mu t}|Du_n|^2
\end{align*}
where we used the growth assumption \ref{DerGrDxHA2} on $D_xH$, Assumption \ref{DerGrDxMHA} and the uniform ellipticity of $H$. Now we can choose $\mu = \mu(\|Df \|_{\infty}, C_{D_xH}, \Lambda^-) >0$ such that the right-hand side of the above expression is bounded by above and, by the maximum principle for parabolic equations we get that $\|Du_n\|_{\infty} \leq C$ for some $C = C(\|Dg' \|_{\infty}, \|Df_2' \|_{\infty}, C_{D_xH}, \Lambda^-) >0$.

Now we let $v_n := \partial_t u_n$. By differentiating the HJB equation with respect to time we get that $v_n$ solves 

\begin{equation*}
\left \{
\begin{array}{ll}
\displaystyle -\partial_t v_n + D_pH.Dv_n + D_MH.D^2v_n = -\partial_tH + \partial_t f'^{(n)}_2 & \mbox{in } [0,T]\times \R^d\\
v_n(T,x) = H(T,x,Dg',D^2g') - f'^{(n)}_2(T,x) & \mbox{in } \R^d.
\end{array}
\right.
\end{equation*}
Fix $(t_0,x_0) \in [0,T] \times \R^d$ and consider a weak solution $m_n \in \mathcal{C}([t_0,T] , \mathcal{P}_2(\R^d))$ to the adjoint equation 

\begin{equation*}
\left \{
\begin{array}{ll}
\displaystyle \partial_tm_n - \mdiv(D_pHm_n) + \sum_{i,j=1}^d \partial_{i,j}^2 ((D_MH)_{i,j}m) = 0  & \mbox{in } [t_0,T]\times \R^d\\
\displaystyle m_n(t_0) = \delta_{x_0}
\end{array}
\right.
\end{equation*}
Integrating $v_n$ against $m_n$ gives, after integration by part and reorganizing the terms:

\begin{align*}
v_n(t_0,x_0) &=\int_{\R^d} \left[H(T,x,Dg',D^2g')-f'^{(n)}_2(T,x) \right]dm^n(T)(x) \\
&- \int_{t_0}^T \int_{\R^d} \partial_tH(t,x,Du_n,D^2u_n) dm^n(t)(x)dt + \int_{t_0}^T \int_{\R^d} \partial_t f'^{(n)}_2(t,x)dm^n(t)(x)dt
\end{align*}
But, again by integration by part, we have

\begin{align*} 
\int_{t_0}^T &\int_{\R^d} \partial_t f'^{(n)}_2(t,x)dm^n(t)(x)dt = \int_{\R^d} f'^{(n)}_2(T,x)dm(T)(x) - f'^{(n)}_2(t_0,x_0) \\
&+ \int_{t_0}^T \int_{\R^d} \left[D_pH(t,x,Du_n,D^2u_n).D f'^{(n)}_2 (t,x) + D_MH(t,x,Du_n,D^2u_n).D^2 f'^{(n)}_2(t,x) \right] dm^n(t)(x)dt
\end{align*}
and we can conclude, using the growth assumption on $D_pH$, Assumption \ref{DerGrDpHA}, and the boundness of $\partial_MH$ and $\partial_tH$, that $ \|\partial_t u_n \|_{\infty} \leq C$ for some $C >0$ depending only on $\|Du_n \|_{\infty}$, $\|f'_2 \|_{\infty}$, $\|Df'_2 \|_{\infty}$, $\|D^2f_2' \|_{\infty}$, $\|Dg' \|_{\infty}$, $\|D^2g' \|_{\infty}$ but not on $\|\partial_t f'^{(n)}_2 \|_{\infty}$.

Combining the two above estimates, we can use the stability of viscosity solutions to show that $u_n$ converges locally uniformly to $u$ and that $u$ is therefore a globally Lipschitz function. Following the discussion at the beginning of this section this is enough to conclude that $u$ belongs to $\mathcal{C}_b^{\frac{3+\alpha}{2},3+\alpha}([0,T] \times \R^d)$.

\end{proof}

\appendix

\section{Appendix}

\label{sec: Appendix}

\label{sec: Von-Neumann Theorem and Proof of the dual representation for $H_1^*$ and $H_2^*$}

\noindent Since it appears twice in our article and in particular in the proof of Theorem \ref{Duality} we recall the statement of the Von-Neumann theorem we are using. The statement and proof can be found in the Appendix of \cite{Orrieri2019} and in a slightly different setting, in \cite{Simons1998}.

\begin{theorem}(Von Neumann)
\label{VN}
Let $\mathbb{A}$ and $\mathbb{B}$ be convex sets of some vector spaces and suppose that $\mathbb{B}$ is endowed with some Hausdorff topology. Let $\cL$ be a function satisfying :

$$ a \rightarrow \cL (a,b) \mbox{      is concave in }\mathbb{A} \mbox{ for every }b \in \mathbb{B}, $$
$$ b \rightarrow \cL (a,b) \mbox{      is convex in }\mathbb{B} \mbox{ for every }a \in \mathbb{A}. $$

\noindent Suppose also that there exists $a_* \in \mathbb{A}$ and $C_* > \sup_{a \in \mathbb{A} } \inf_{b \in \mathbb{B}} \cL(a,b)$ such that :

$$ \mathbb{B}_*:= \left \{ b \in \mathbb{B}, \cL (a_*,b) \leq C_* \right \} \mbox{      is not empty and compact in } \mathbb{B}, $$
$$ b \rightarrow \cL(a,b) \mbox{        is lower semicontinuous in } \mathbb{B}_* \mbox{   for every  } a \in \mathbb{A}. $$

\noindent Then,

$$ \min_{b\in \mathbb{B} } \sup_{a \in \mathbb{A} } \cL(a,b) = \sup_{ a \in \mathbb{A} } \inf_{b \in \mathbb{B} } \cL(a,b). $$

\end{theorem}

\begin{remark} The fact that the infimum in the “$\inf \sup$" problem is in fact a minimum is part of the theorem.
\end{remark}

\paragraph{Acknowledgments} The author thanks the anonymous referees for their comments and careful proofreading of the paper.

The author was partially supported by the ANR (Agence Nationale de la Recherche) project ANR-16-CE40-0015-01 on Mean Field Games. Part of this research was performed while the author was visiting the Institute for Mathematical and Statistical Innovation (IMSI), which is supported by the National Science Foundation (Grant No. DMS-1929348).

The author wishes to thank Professor Pierre Cardaliaguet (Paris Dauphine) for fruitful discussions all along this work.

\bibliographystyle{plain}
\bibliography{/Users/sam/Documents/Bibtex/DraftPaper1.bib}

\begin{thebibliography}{10}

\bibitem{Achdou2010}
Yves Achdou and Italo {Capuzzo Dolcetta}.
\newblock {Mean field games: Numerical methods}.
\newblock {\em SIAM Journal on Numerical Analysis}, 48(3):1136--1162, 2010.

\bibitem{Alvarez1997}
O.~Alvarez, J.~M. Lasry, and P.~L. Lions.
\newblock {Convex viscosity solutions and state constraints}.
\newblock {\em Journal des Mathematiques Pures et Appliquees}, 76(3):265--288,
  1997.

\bibitem{Ambrosio2000}
Luigi Ambrosio, Nicola Fusco, and Diego Pallara.
\newblock {\em {Functions of Bounded Variation and Free Discontinuity
  Problems}}.
\newblock Oxford Mathematical Monographs, 2000.

\bibitem{Ambrosio2005}
Luigi Ambrosio, Nicola Gigli, and Giuseppe Savare.
\newblock {\em {Gradient Flows in metric spaces and in the space of probability
  measures}}.
\newblock Birkh{\"{a}}user Basel, 2005.

\bibitem{Armstrong2018}
Scott Armstrong and Pierre Cardaliaguet.
\newblock {Stochastic homogenization of quasilinear Hamilton-Jacobi equations
  and geometric motions}.
\newblock {\em Journal of the European Mathematical Society}, 20(4):797--864,
  2018.

\bibitem{Barles1991}
Guy Barles.
\newblock {A weak Bernstein method for fully nonlinear elliptic equations}.
\newblock {\em Differential and Integral Equations}, 4:241--262, 1991.

\bibitem{Benamou2000}
Jean-David Benamou and Yann Brenier.
\newblock {A computational fluid mechanics solution to the Monge-Kantorovich
  mass transfer problem}.
\newblock {\em Numerische Mathematik}, 84(3):375--393, 2000.

\bibitem{Blaquiere1992}
A.~Blaqui{\`{e}}re.
\newblock {Controllability of a Fokker-Planck Equation, The Schr{\"{o}}dinger
  System, and a Related Stochastic Optimal Control}.
\newblock {\em Dynamics and Control}, 2:235--253, 1992.

\bibitem{Bouchard2009}
Bruno Bouchard, Romuald Elie, and Cyril Imbert.
\newblock {Optimal control under stochastic target constraints}.
\newblock {\em SIAM Journal on Control and Optimization}, 48(5):3501--3531,
  2009.

\bibitem{Bouchard2010}
Bruno Bouchard, Romuald Elie, and Nizar Touzi.
\newblock {Stochastic Target Problems with Controlled Loss}.
\newblock {\em SIAM Journal on Control and Optimization}, 48(5):3123--3150,
  2010.

\bibitem{Bourgoing2004}
Mariane Bourgoing.
\newblock {C1,$\beta$ regularity of viscosity solutions via a
  continuous-dependence result}.
\newblock {\em Advances in Differential Equations}, 9(3-4):447--480, 2004.

\bibitem{Briani2018}
Ariela Briani and Pierre Cardaliaguet.
\newblock {Stable solutions in potential mean field game systems}.
\newblock {\em Nonlinear Differential Equations and Applications}, 25(1):1--26,
  2018.

\bibitem{Cardaliaguet2015}
Pierre Cardaliaguet, P.~Jameson Graber, Alessio Porretta, and Daniela Tonon.
\newblock {Second order mean field games with degenerate diffusion and local
  coupling}.
\newblock {\em Nonlinear Differential Equations and Applications},
  22(5):1287--1317, 2015.

\bibitem{Carmona2018a}
Ren{\'{e}} Carmona and Fran{\c{c}}ois Delarue.
\newblock {\em {Probabilistic theory of mean field games with applications.
  I}}.
\newblock Springer International Publishing, 2018.

\bibitem{Chow2020}
Yuk~Loong Chow, Xiang Yu, and Chao Zhou.
\newblock {On Dynamic Programming Principle for Stochastic Control Under
  Expectation Constraints}.
\newblock {\em Journal of Optimization Theory and Applications},
  185(3):803--818, 2020.

\bibitem{Crandall1992}
Michael~G. Crandall, Hitoshi Ishii, and Pierre-Louis Lions.
\newblock {User's guide to viscosity solutions of second order partial
  differential equations}.
\newblock {\em Bulletin of the American Mathematical Society}, 27(1):1--67,
  1992.

\bibitem{Ekeland1999a}
Ivar Ekeland and Roger T{\'{e}}mam.
\newblock {\em {Convex Analysis and Variational Problems}}.
\newblock Society for Industrial and Applied Mathematics, 1999.

\bibitem{ElKaroui1987}
Nicole {El Karoui}, Du'hŪŪ Nguyen, and Monique Jeanblanc-Picqu{\'{e}}.
\newblock {Compactification methods in the control of degenerate diffusions:
  existence of an optimal control}.
\newblock {\em Stochastics}, 20(3):169--219, 1987.

\bibitem{Figalli2008}
Alessio Figalli.
\newblock {Existence and uniqueness of martingale solutions for SDEs with rough
  or degenerate coefficients}.
\newblock {\em Journal of Functional Analysis}, 254(1):109--153, 2008.

\bibitem{Fleming2006}
Wendell~H. Fleming and H.~Mete Soner.
\newblock {\em {Controlled Markov Processes and Viscosity Solutions}}.
\newblock Springer-Verlag New York, 2006.

\bibitem{Fleming1989}
Wendell~H. Fleming and Domokos Vermes.
\newblock {Convex duality approach to the optimal control of diffusions}.
\newblock {\em SIAM Journal on Control and Optimization}, 27(5):1136--1155,
  1989.

\bibitem{Follmer1999}
Hans F{\"{o}}llmer and Peter Leukert.
\newblock {Quantile hedging}.
\newblock {\em Finance and Stochastics}, 3(3):251--273, 1999.

\bibitem{Guo2020}
Ivan Guo, Nicolas Langren{\'{e}}, Gr{\'{e}}goire Loeper, and Wei Ning.
\newblock {Portfolio optimization with a prescribed terminal wealth
  distribution}.
\newblock 2020.

\bibitem{Guo2019}
Ivan Guo, Gr{\'{e}}goire Loeper, and Shiyi Wang.
\newblock {Calibration of local-stochastic volatility models by optimal
  transport}.
\newblock {\em arXiv}, (1999):1--28, 2019.

\bibitem{Imbert2006}
Cyril Imbert.
\newblock {Convexity of solutions and C1, 1 estimates for fully nonlinear
  elliptic equations}.
\newblock {\em Journal des Mathematiques Pures et Appliquees}, 85(6):791--807,
  2006.

\bibitem{Ishii1990}
Hitoshi Ishii and Pierre-Louis Lions.
\newblock {Viscosity solutions of fully nonlinear second-order elliptic partial
  differential equations}.
\newblock {\em Journal of Differential Equations}, 83:26--78, 1990.

\bibitem{Karatzas1991}
Ioannis Karatzas and Steven~E. Shreve.
\newblock {\em {Brownian Motion and Stochastic Calculus}}.
\newblock Springer edition, 1991.

\bibitem{Krylov1980}
N.V. Krylov.
\newblock {\em {Controlled Diffusion Processes}}.
\newblock Springer-Verlag Berlin Heidelberg, 1980.

\bibitem{Lacker2015}
Daniel Lacker.
\newblock {Mean field games via controlled martingale problems: Existence of
  Markovian equilibria}.
\newblock {\em Stochastic Processes and their Applications}, 125(7):2856--2894,
  2015.

\bibitem{Ladyzenskaja1968}
O.A. Ladyzenskaja, V.A. Solonnikov, and N.N. Uralceva.
\newblock {\em {Linear and non-linear equations of parabolic type}}.
\newblock 1968.

\bibitem{Lasry2007}
Jean-Michel Lasry and Pierre-Louis Lions.
\newblock {Mean field games}.
\newblock {\em Japanese Journal of Mathematics}, 2(1):229--260, 2007.

\bibitem{Lions2005}
Pierre-Louis Lions and Panagiotis Souganidis.
\newblock {Homogenization of degenerate second-order PDE in periodic and almost
  periodic environments and applications}.
\newblock {\em Annales De L Institut Henri Poincare-analyse Non Lineaire},
  22:667--677, 2005.

\bibitem{Mikami2015}
Toshio Mikami.
\newblock {Two End Points Marginal Problem by Stochastic Optimal
  Transportation}.
\newblock {\em SIAM J. CONTROL OPTIM.}, 53(4):2449--2461, 2015.

\bibitem{Mikami2006}
Toshio Mikami and Mich{\`{e}}le Thieullen.
\newblock {Duality theorem for the stochastic optimal control problem}.
\newblock {\em Stochastic Processes and their Applications},
  116(12):1815--1835, 2006.

\bibitem{Milgrom2002}
B.Y.~Paul Milgrom and Ilya Segal.
\newblock {Envelope theorems for arbitrary choice sets}.
\newblock {\em Econometrica}, 70(2):583--601, 2002.

\bibitem{Orrieri2019}
Carlo Orrieri, Alessio Porretta, and Giuseppe Savar{\'{e}}.
\newblock {A variational approach to the mean field planning problem}.
\newblock {\em Journal of Functional Analysis}, 2019.

\bibitem{Pfeiffer2020}
Laurent Pfeiffer.
\newblock {Optimality conditions in variational form for non-linear constrained
  stochastic control problems}.
\newblock {\em Mathematical Control and Related Fields}, 10(3):493--526, 2020.

\bibitem{Pfeiffer2020a}
Laurent Pfeiffer, Xiaolu Tan, and Yulong Zhou.
\newblock {Duality and approximation of stochastic optimal control problems
  under expectation constraints}.
\newblock 2020.

\bibitem{Rachev1998}
Svetlozar~T. Rachev and Ludger R{\"{u}}schendorf.
\newblock {\em {Mass Transportation Problems - Volume 1: Theory}}.
\newblock Springer-Verlag New York, 1998.

\bibitem{Simons1998}
Stephen Simons.
\newblock {\em {Minimax and Monoticity}}.
\newblock Springer-Verlag Berlin Heidelberg, 1998.

\bibitem{Soner2002}
H.~Mete Soner and Nizar Touzi.
\newblock {Dynamic programming for stochastic target problems and geometric
  flows}.
\newblock {\em Journal of the European Mathematical Society}, 4(3):201--236,
  2002.

\bibitem{Stroock1997}
Daniel~W. Stroock and S.~R.~Srinivasa Varadhan.
\newblock {\em {Multidimensional Diffusion Processes}}.
\newblock Springer-Verlag Berlin Heidelberg, 1997.

\bibitem{Tan2013}
Xiaolu Tan and Nizar Touzi.
\newblock {Optimal transportation under controlled stochastic dynamics}.
\newblock {\em Annals of Probability}, 41(5):3201--3240, 2013.

\bibitem{Trevisan2016}
Dario Trevisan.
\newblock {Well-posedness of multidimensional diffusion processes with weakly
  differentiable coefficients}.
\newblock {\em Electronic Journal of Probability}, 21:1--42, 2016.

\bibitem{Villani2003}
C{\'{e}}dric Villani.
\newblock {\em {Topics in Optimal Transportation}}.
\newblock Graduate Studies in Mathematics, 2003.

\bibitem{Villani2007}
C{\'{e}}dric Villani.
\newblock {\em {Optimal Transport Old and New}}.
\newblock Springer-Verlag Berlin Heidelberg, 2009.

\bibitem{Wang1992a}
Lihe Wang.
\newblock {On the regularity theory of fully nonlinear parabolic equations: I}.
\newblock {\em Communications on Pure and Applied Mathematics}, 45(1):27--76,
  1992.

\bibitem{Wang1992b}
Lihe Wang.
\newblock {On the regularity theory of fully nonlinear parabolic equations:
  II}.
\newblock {\em Communications on Pure and Applied Mathematics}, 45(2):141--178,
  1992.

\bibitem{Yong1999}
Jiongmin Yong and Xun~Yu Zhou.
\newblock {\em {Stochastic Controls - Hamiltonian systems and HJB equations}}.
\newblock Springer-Verlag New York, 1999.

\end{thebibliography}

\end{document}